\documentclass[12pt]{amsart}
\usepackage{amsfonts,amsmath,amssymb,amscd,mathrsfs}
\usepackage{xcolor}
\usepackage{verbatim}
\usepackage{comment}
\usepackage{enumitem}
\setlist[enumerate,1]{label=\textup{(\roman*)}}
\usepackage[colorlinks=true,
linkcolor=blue,
citecolor=blue,
urlcolor=blue]{hyperref}
\newcommand{\newsection}[1]{\setcounter{equation}{0} \section{#1}}
\newtheorem{theorem}{Theorem}[section]
\newtheorem{lemma}[theorem]{Lemma}
\newtheorem{corollary}[theorem]{Corollary}
\newtheorem{proposition}[theorem]{Proposition}
\theoremstyle{definition}

\newtheorem{example}[theorem]{Example}
\newtheorem{remark}[theorem]{Remark}
\theoremstyle{remark}
\newtheorem{question}[theorem]{Question}

\newcommand{\D}{\mathbb{D}}

\newcommand{\R}{\mathbb{R}}
\newcommand{\C}{\mathbb{C}}
\newcommand{\T}{\mathbb{T}}
\newcommand{\wt}{\widetilde}
\newcommand{\clb}{\mathcal{B}}
\newcommand{\clh}{\mathcal{H}}
\newcommand{\clk}{\mathcal{K}}
\newcommand{\clr}{\mathcal{R}}
\newcommand{\cli}{\mathcal{I}}
\newcommand{\clt}{\mathcal{T}}
\newcommand{\cls}{\mathcal{S}}
\newcommand{\clc}{\mathcal{C}}
\newcommand{\cld}{\mathcal{D}}
\newcommand{\cle}{\mathcal{E}}
\newcommand{\clz}{\mathcal{Z}}
\newcommand{\clm}{\mathcal{M}}
\newcommand{\clv}{\mathcal{V}}
\newcommand{\clw}{\mathcal{W}}
\newcommand{\vp}{\varphi}
\newcommand{\ran}{\mathrm{ran}}
\newcommand{\ip}[2]{\left\langle #1,#2\right\rangle}
\newcommand{\norm}[1]{\left\|#1\right\|}

\numberwithin{equation}{section}
\usepackage{mathtools}
\newcommand{\supp}{\operatorname{supp}}
\newcommand{\spn}{\operatorname{span}}
\newcommand{\bm}{\mathbf}
\newcommand{\Sig}{\Sigma_{\theta}}
\newcommand{\NI}{\noindent}

\begin{document}
	
	\title[Algebraic properties of TTO $+$ THO]
	{Algebraic properties of Truncated Toeplitz $+$ Truncated Hankel operators}
	
	\author[Babbar]{Kritika Babbar}
	\address{Department of Mathematics, Indian Institute of Technology Roorkee, Roorkee--247\,667, Uttarakhand, India}
	\email{kritikababbariitr@gmail.com, kritika@ma.iitr.ac.in}
	
	\author[Majee]{Satyabrata Majee}
	\address{School of Mathematical and Statistical Sciences, North-West University, Potchefstroom, 2520, South Africa}
	\email{majeesatyabrata@gmail.com, 56337817@mynwu.ac.za}
	
	\author[Maji]{Amit Maji}
	\address{Department of Mathematics, Indian Institute of Technology Roorkee, Roorkee--247\,667, Uttarakhand, India}
	\email{amit.maji@ma.iitr.ac.in, amit.iitm07@gmail.com}
	

\begin{abstract}		 		
In this paper, we study the algebraic properties of $\cls_\theta$ --- the class of truncated Toeplitz plus truncated Hankel operators --- on the model space $\clk_\theta$. We also give a complete characterization of the intersection class of truncated Toeplitz and truncated Hankel operators. As a by-product of our findings, we give a complete characterization for $S_\theta$, where $\theta$ is a finite Blaschke product and also obtained the commutant of $S_\theta+S_\theta^{*}$ in that case.
\end{abstract}
\maketitle

\newsection{Introduction}\label{sec:intro}
	
Let $\theta$ be a non-constant inner function on the unit disc $\D$ and let
$\clk_\theta=H^2\ominus\theta H^2$ be the associated model space, with compressed shift
$S_\theta f=P_{\clk_\theta}(zf)$. Sarason \cite{Sarason} introduced the truncated Toeplitz operators on $\clk_\theta$ and characterized them by an operator equation involving $S_\theta$. This work initiated a broad study of the various aspects of the subject, including the spectral and algebraic properties of truncated Toeplitz operators. In \cite{RB}, Bessonov gave a complete characterization of finite rank truncated Toeplitz operators. For a detailed study of truncated Toeplitz operators, one can refer to \cite{CIMA-ROSS-WOGEN, G-ROSS-W, GMR, MaZheng} and the references therein.

The notion of truncated Hankel operators has been developed later than that of truncated Toeplitz operators. Recently, Gu and Ma developed a systematic theory of truncated Hankel operators on model spaces in \cite{GuMa}. They established an analogous characterization for truncated Hankel operators and also investigated their fundamental and algebraic properties. Later on, the algebraic relationship and the products involving truncated Toeplitz and truncated Hankel operators have been discussed in \cite{XZ}. Suppose
$\clt_\theta$ and $\clh_\theta$ denote the bounded truncated Toeplitz and truncated Hankel operators on $\clk_{\theta}$. The present paper is concerned with the algebraic structure of the sum
	\[
	\cls_\theta=\clt_\theta+\clh_\theta,
	\]
the model-space analogue of the classical class of Toeplitz-plus-Hankel matrices \cite{BBB}. The algebraic and characterization aspects of Toeplitz $+$ Hankel operators are discussed by various mathematicians (cf.~ \cite{DDS, Gu, Sang} and references therein). However, limited work has been done in the context of truncated Toeplitz $+$ truncated Hankel operators on the model space (cf.~\cite{BM}).

We intend to study the structure of $\cls_\theta$, which is controlled by the intersection
	\[
	\cli_\theta=\clt_\theta\cap\clh_\theta .
	\]
Indeed, if $T \in \cls_\theta$ and $T=A+B=A'+B'$ are two decompositions of the same operator, with $A, A'\in\clt_\theta$ and $B, B'\in\clh_\theta$, then $A'-A=B-B'$ lies in $\cli_\theta$. In this way, $\cli_\theta$ measures the non-uniqueness of the decomposition. Furthermore, other algebraic properties such as the existence of a self-adjoint decomposition and an intrinsic membership test for $\cls_\theta$ are made available by $\cli_\theta$. A displacement calculus for $\cls_\theta$ and a symbol-free characterization of $\cls_\theta$ have been developed recently in \cite{BM}.  Moreover, they provide a complete description of $\cli_\theta$ for $\theta=z^n$ and prove that $\dim \cli_\theta=2$ in that case. Indeed, for $\theta=z^n$, $\clt_\theta$ and $\clh_\theta$ are the spaces of $n \times n$ Toeplitz and Hankel matrices, respectively. Thus, $\cli_\theta$, the space of such $n \times n$ matrices which are both Toeplitz and Hankel, is spanned by the following matrices (cf.~\cite[Remark 5.9]{BM}):
	\[
	(E_1)_{ij}=\begin{cases}
		1, & i+j \text{ even}\\
		0, & i+j \text{ odd}
	\end{cases}
	\quad \text{and}\quad
	(E_2)_{ij}=\begin{cases}
		0, & i+j \text{ even}\\
		1, & i+j \text{ odd}
	\end{cases}.
	\]

In this paper, we substantially generalize this result. We prove that $\dim \cli_\theta=2$ whenever $\theta$ has an angular derivative in the sense of Carath\'{e}odory (ADC) at both $1$ and $-1$. In particular, this holds for every finite Blaschke product. More generally, for every inner function $\theta$ with $\dim\clk_\theta\ge3$, 
we obtain the complete determination
	\[
	\dim\cli_\theta=\#\{\zeta\in\{1,-1\}:\theta\ \text{has an ADC at}\ \zeta\}\in\{0,1,2\},
	\]
so that each of the values $0,1,2$ occurs (see Section~\ref{sec:tri}).
	
	The truncated Toeplitz and truncated Hankel operators are associated with the following maps on $\clb(\clk_\theta)$:
	\[
	D_T(X)=X-S_\theta X S_\theta^* \qquad\text{and}\qquad D_H(X)=X-S_\theta^* X S_\theta^*.
	\]
	Bhuia and Majee \cite{BM} considered the second-order map $D_HD_T$ and proved that for a bounded operator $X \in \cls_\theta$, $D_HD_T(X)$ belongs to
	\[
	\clm:=\{f \otimes k_0^\theta +k_0^\theta \otimes g +u \otimes S_\theta k_0^\theta +\wt{k}_0^\theta \otimes v: f,g,u,v \in \clk_\theta\}.
	\]
	The natural question lies in the converse:
	\begin{question}
		If $X \in \clb(\clk_\theta)$ satisfies $D_H D_T(X) \in \clm$, does $X \in \cls_\theta$?
	\end{question}
	
In this paper, we give an affirmative answer to this question when $\clk_\theta$ is a finite-dimensional model space. Furthermore, as a by-product of our findings, we aim to explore algebraic and structural properties of $\cls_\theta$. By utilizing the complete description of $\cli_\theta$, we provide a new characterization of $\cls_\theta$ for a finite Blaschke product $\theta$. Another important object in this analysis is the operator $S_\theta +S_\theta^*$. In \cite{Gu}, Gu investigated the commutation relations of Toeplitz $+$ Hankel operators with $S+S^*$ (where $S$ is the unilateral shift on the Hardy space) and showed that any bounded operator commuting with $S+S^*$ must be of type \emph{Toeplitz + Hankel}. Our next goal is to provide the analogous result in the framework of finite-dimensional model spaces. More precisely, we prove that any bounded operator commuting with $S_\theta +S_\theta^*$ necessarily belongs to $\cls_\theta$. Additionally, we explicitly determine the commutant of $S_\theta+S_\theta^*$ for the finite-dimensional model space $\clk_\theta$.

Turning back to algebraic properties, Sang \cite{Sang} obtained a necessary and sufficient condition for the product of two \emph{Toeplitz + Hankel} operators to be another \emph{Toeplitz + Hankel} operator. In our paper,  we establish a symbol-free criterion for their truncated counterparts. Specifically, we obtain a necessary and sufficient condition under which the product of two \emph{truncated Toeplitz + truncated Hankel} operators remains \emph{truncated Toeplitz + truncated Hankel}. Finally, we conclude our paper with some remarks related to the extension of these concepts for the case of infinite-dimensional model space.

The organization of the paper is as follows: In Section \ref{sec:prelim}, we provide relevant definitions and results that will be used throughout the paper. Section \ref{sec:upper} provides a complete description of the intersection class $\cli_\theta$ under suitable conditions. In Section \ref{sec:tri}, we explore the structure of $\cli_\theta$ for an arbitrary inner function without any restriction. Section \ref{sec:alg} focuses on the algebraic implications of the structure of $\cli_\theta$. In Section \ref{sec:comm}, we provide a commutator-based characterization of $\cls_\theta$ and also give necessary and sufficient conditions under which the product of two elements of $\cls_\theta$  again lies in $\cls_\theta$ when $\theta$ is a finite Blaschke product. Finally, we give concluding remarks in Section \ref{sec:Conclusion}.

	\newsection{Preliminaries and basic results}\label{sec:prelim}

In this section, we recall the preliminary definitions and results that will be repeatedly invoked in the sequel. Firstly, for $f, g \in \clk_\theta$, recall the rank-one operator $f\otimes g: \clk_{\theta} \rightarrow \clk_{\theta}$ acts by
	\[
	(f\otimes g)h=\ip{h}{g}f, \qquad  h \in \clk_\theta, 
	\]
	and for $X, Y \in \clb(\clk_\theta)$ and $\alpha,\beta \in \C$, we have:
	\begin{equation}\label{eq:tensor}
		X(f\otimes g)Y=(Xf)\otimes(Y^{*}g),\quad (\alpha f)\otimes (\beta g)=\alpha\bar{\beta}\,(f\otimes g), \quad [X, Y] =XY -YX.
	\end{equation}
	Throughout the paper, we adopt the standard conventions commonly used in the model-space literature \cite{GMR,GuMa,Nik,Sarason}.

Let $\D=\{z:|z|<1\}$ and $\T=\partial\D$ be the open unit disc and the unit circle in the complex plane $\C$, respectively. Let $L^2$ for the space of square-integrable functions with respect to the normalised Lebesgue measure on $\T$, and the Hardy space $H^2$ is the closed subspace of $L^2$ consisting of all such functions of $L^2$ whose negative Fourier coefficients vanish. We write $P_+$ for the orthogonal projection of $L^2$ onto $H^2$, and $P_-=I-P_+$ for the orthogonal projection of $L^2$ onto $H^2_-=L^2\ominus H^2$. Let $L^\infty$ be the Banach algebra of all essentially bounded functions on $\T$ and $H^\infty$ denote the space of all bounded analytic functions on $\D$. For any $\vp \in L^\infty$, the \emph{Laurent operator} on $L^2$, denoted by $M_\vp$,  is defined as
	\[
	M_\vp(f)=\vp f, \qquad (f \in L^2).
	\]

A function $\theta \in H^\infty$ is said to be inner if $|\theta(e^{it})|=1$ a.e.\ on $\T$. For $\theta$ inner, the corresponding \emph{model space} is $\clk_\theta=H^2\ominus\theta H^2$. The reproducing kernel and conjugate kernel at 
$\lambda\in\D$ are given by
	\[
	k_\lambda^\theta(z)=\frac{1-\overline{\theta(\lambda)}\,\theta(z)}{1-\overline{\lambda}z},
	\qquad
	\wt{k}_\lambda^\theta(z)=\frac{\theta(z)-\theta(\lambda)}{z-\lambda},
	\]
respectively. In particular, $k_0^\theta=1-\overline{\theta(0)}\,\theta$.  When $\theta$ is a finite Blaschke product, we write $n=\deg\theta=\dim\clk_\theta$. The compressed shift $S_\theta$ on $\clk_\theta$ is defined as
	\[
	S_\theta f=P_{\clk_\theta}(zf), \qquad (f \in \clk_\theta) 
	\]
with adjoint $S_\theta^* f=P_{\clk_\theta}(\overline{z}f)$ where $P_{\clk_\theta}$ is the orthogonal projection onto $\clk_\theta$.   For a non-constant inner $\theta$, $S_\theta$ is a completely non-unitary contraction of class $C_{00}$	\cite[Ch.~III]{SNFBK}, i.e., $S_\theta^{n}f \to 0$ and $(S_\theta^{*})^{n}f \to 0$ as $n \to \infty$ for every
$f \in \clk_\theta$. In what follows, the operator $C_\theta: L^2 \rightarrow L^2$ defined by
	\[
	(C_\theta f)(z)=\theta(z)\,\overline{z}\,\overline{f(z)},\qquad (z\in\T),
	\]
	is a conjugation (antilinear isometric involution)  map which bijectively maps $\clk_\theta$ to itself and satisfies
	$\wt{k}_\lambda^\theta=C_\theta k_\lambda^\theta$.
We record some basic facts related to $S_\theta$ from \cite[Lemma 2.4]{Sarason}:
	\begin{align}
		I-S_\theta S_\theta^*&=k_0^\theta\otimes k_0^\theta, \label{eq:s1}\\
		I-S_\theta^* S_\theta&=\wt{k}_0^\theta\otimes\wt{k}_0^\theta, \label{eq:s2}\\
		S_\theta^* k_0^\theta&=-\overline{\theta(0)}\,\wt{k}_0^\theta, \label{eq:s3}\\
		S_\theta\wt{k}_0^\theta&=-\theta(0)\,k_0^\theta, \label{eq:s4}
	\end{align}
together with the antilinear relation $S_\theta^*=C_\theta S_\theta C_\theta$.

Finally, we recall the description of the spectrum $\sigma(S_\theta)$ of $S_\theta$. For an inner function $\theta=B\,s_\mu$, where $B$ is a Blaschke product with zero set 
$\{a_j\}$ and $s_\mu$ is a singular inner function with measure $\mu$, the spectrum
	\begin{equation}\label{eq:specthm}
		\sigma(S_\theta)=\overline{\{a_j\}}\cup\supp\mu
		=\{a_j\}\cup\bigl\{\zeta\in\T:\theta\ \text{has a singularity at}\ \zeta\bigr\}
	\end{equation}
	\cite[Ch.~5]{GMR}, \cite[Ch.~3]{Nik}. Note that, since the zeros of a
	non-zero analytic function are isolated in $\D$, the point $0 \in \sigma(S_\theta)$ if and
	only if $\theta(0)=0$; hence
	\begin{equation}\label{eq:invertible}
		\theta(0)\ne0\ \iff \ 0\notin\sigma(S_\theta)\ \iff\
		S_\theta\ \text{and}\ S_\theta^{*}\ \text{are invertible.}
	\end{equation}
	
\subsection{Angular derivatives} Let $f: \D \rightarrow \D $ be an analytic function and $\zeta \in \T$ such that $f'$ has a non-tangential limit at $\zeta$  (see \cite[Chapter 2]{GMR} for details), the quantity
	\[
	f'(\zeta)=\angle \lim_{z \to \zeta} \frac{f(z)-f(\zeta)}{z-\zeta}=\angle \lim_{z\to \zeta}f'(z)
	\]
	is called the \emph{angular derivative} of $f$ at $\zeta$. In addition, we say that $f$ has an \emph{angular derivative in the sense of Carath\'{e}odory}  (ADC) at $\zeta \in \T$ if the following conditions hold:
	\begin{enumerate}
		\item $f$ has an angular derivative at $\zeta$,
		\item $|f(\zeta)|=1$.
	\end{enumerate}
	
	A characterization for $\theta$ to have an ADC at $\zeta \in \T$ is given by Julia and Carath\'{e}odory. One can refer to \cite[Chapter 7]{GMR} for the following:
	\begin{lemma}\label{lem:Julia}
		Let $\theta$ be an inner function and $\zeta \in \T$. Then the following are equivalent:
		\begin{enumerate}
			\item $l:=\liminf\limits_{z\to \zeta} \frac{1-|\theta(z)|}{1-|z|}<\infty$.
			\item There exists an $\alpha \in \T$ such that
			\[
			\frac{1-\bar{\alpha}\theta}{1-\bar{\zeta}z}\in \clk_\theta.
			\]
			\item Each $f \in \clk_\theta$ has a non-tangential limit at $\zeta$.
			\item $\theta$ has an ADC at $\zeta$.
		\end{enumerate}
		Moreover, under the conditions above,  $l=\angle\lim\limits_{z\to \zeta} \frac{1-|\theta(z)|}{1-|z|}=|\theta'(\zeta)|=k_\zeta^\theta(\zeta)$ and $\alpha=\theta(\zeta)$.
	\end{lemma}

\subsection{Truncated Toeplitz and Hankel operators}
For a non-constant inner function $\theta$, the space $\clk_\theta \cap H^\infty$ is dense in $\clk_\theta$. For $\vp\in L^2$, the truncated Toeplitz operator $A_\vp $ is the densely defined operator on $\clk_\theta$ given by 
	\[
	A_{\vp}f=P_{\clk_\theta} M_{\vp}f, \quad (f \in \clk_\theta \cap H^\infty).
	\]
We say that $A_\vp$ is bounded if it extends to a bounded operator on $\clk_\theta$, and we denote the space of all bounded truncated Toeplitz operators on $\clk_\theta$ by $\clt_\theta$. It is well-known that $A_{\vp}^*=A_{\overline{\vp}}$ and thus $\clt_\theta$ is closed under adjoints \cite{Sarason}. All the truncated Toeplitz operators are $C_\theta$-symmetric \cite{GarciaPutinar}, that is,
	\[
	A_{\vp}^*=C_\theta A_{\vp}C_\theta.
	\]
	
	For any $\vp \in L^2$, define the truncated Hankel operator $B_\vp$ on $ \clk_\theta \cap H^\infty$ by
	\[
	B_{\vp}f=P_{\clk_\theta} J\,P_- M_{\vp}f, \quad (f \in \clk_\theta \cap H^\infty)
	\]
	where $J:L^2 \rightarrow L^2$ defined by
	\[
	(Jf)(z)=\overline{z}\,f(\overline{z}),
	\]
	is the \emph{flip} operator. Note that $J$ is	a unitary involution of $L^2$ which maps $H^2$ into $H^2_-$ and $H^2_-$ into $H^2$. We write  $\clh_\theta$ for the set of all bounded truncated Hankel operators (bounded in the same  sense as above). For $\vp \in L^2$, denote $\vp^*=\overline{\vp(\bar{z})}$. Then $B_{\vp}^*=B_{\vp^*}$. In particular, $\clh_\theta$ is closed under taking adjoints.

	The classes in which we are interested  are:
	\[
	\cls_\theta=\clt_\theta+\clh_\theta\qquad \text{and} \qquad \cli_\theta=\clt_\theta\cap\clh_\theta.
	\]
By Sarason's theorem \cite[Theorem 4.1]{Sarason}, $A\in\clb(\clk_\theta)$ is a truncated Toeplitz operator if and only if there exist $\psi,\chi\in \clk_\theta$ such that
	\begin{equation}\label{eq:TTOchar}
		A-S_\theta A S_\theta^*=\psi\otimes k_0^\theta+k_0^\theta\otimes\chi,
	\end{equation}
	and by Gu--Ma's theorem \cite[Theorem 3.1]{GuMa}, $B\in\clb(\clk_\theta)$ is a truncated Hankel
	operator if and only if there exist $f,g\in \clk_\theta$ with
	\begin{equation}\label{eq:THOchar}
		B-S_\theta^* B S_\theta^*=f\otimes k_0^\theta+\wt{k}_0^\theta\otimes g.
	\end{equation}
	Accordingly, we write
	\[
	\clr_T=\{u\otimes k_0^\theta+k_0^\theta\otimes v:u,v\in \clk_\theta\},
	\qquad
	\clr_H=\{x\otimes k_0^\theta+\wt{k}_0^\theta\otimes y:x,y\in \clk_\theta\}.
	\]
	Then both $\clr_T$ and $\clr_H$ are linear subspaces of $\clb(\clk_\theta)$. Following \cite{BM}, we recall the \emph{displacement operators}
	\begin{equation}\label{eq:disp}
		D_T(X)=X-S_\theta X S_\theta^{*},\qquad D_H(X)=X-S_\theta^{*}X S_\theta^{*}\qquad(X\in\clb(\clk_\theta)),
	\end{equation}
	so that, by \eqref{eq:TTOchar}--\eqref{eq:THOchar},
	\[
	A\in\clt_\theta\iff D_T(A)\in\clr_T
	\]
	and
	\[
	B\in\clh_\theta\iff D_H(B)\in\clr_H.
	\]
	The following result related to displacement operators is given by Bhuia and Majee in \cite{BM}.
	
	\begin{lemma}[{\cite[Lemmas 5.3 and 5.4]{BM}}]\label{lem:BMdisp}
		Let $X \in \clb(\clk_\theta)$. Then
		\begin{enumerate}
			\item The maps $D_T$ and $D_H$ are injective.
			\item The commutator $[D_H,D_T]$ satisfies
			\begin{equation}\label{eq:comm}
				[D_H,D_T](X)=k_0^\theta\otimes\bigl(S_\theta^{2}X^{*}k_0^\theta\bigr)
				-\wt{k}_0^\theta\otimes\bigl(S_\theta^{2}X^{*}\wt{k}_0^\theta\bigr).
			\end{equation}
		\end{enumerate}
	\end{lemma}
	
When $\theta$ is a finite Blaschke product of degree $n$, $D_T$ and $D_H$ become bijections, so $\dim\clt_\theta=\dim\clr_T=2n-1$ because $\clr_T$ is the sum of the two
$n$-dimensional spaces $\{u\otimes k_0^\theta\}$ and $\{k_0^\theta\otimes v\}$, intersecting in $\C\,(k_0^\theta\otimes k_0^\theta)$, and likewise 
$\dim\clh_\theta=\dim\clr_H=2n-1$.

Now we collect some linear-independence results used throughout. The following results may be well known to experts; we are giving the proofs for the sake of completeness.
	
\begin{lemma}\label{lem:cyclic}
Let $\theta$ be a non-constant inner function with $\dim\clk_\theta\ge2$. Then the pairs
$\{k_0^\theta,S_\theta k_0^\theta\}$ and $\{k_0^\theta,\wt{k}_0^\theta\}$ are each linearly independent.
\end{lemma}

\begin{proof}
		Suppose first that $S_\theta k_0^\theta=\lambda k_0^\theta$ for some $\lambda\in\C$. Applying
		$S_\theta$ repeatedly gives $S_\theta^{m}k_0^\theta=\lambda^{m}k_0^\theta$, so
		$S_\theta^{m}k_0^\theta\in\spn\{k_0^\theta\}$ for every $m\ge0$. Hence any
		$f\in \spn\{k_0^\theta\}^\perp$ is orthogonal to the whole orbit $\{S_\theta^{m}k_0^\theta\}_{m\ge0}$,
		and therefore to its closed linear span. Since $k_0^\theta$ is a cyclic vector for $S_\theta$
		\cite[Lemma 2.3]{Sarason}, that is,
		\[
		\overline{\operatorname{span}} \{S_\theta^n k_0^\theta: n \geq 0\}=\clk_\theta,
		\]
		so $f=0$. Thus
		$\clk_\theta=\spn\{k_0^\theta\}$ is one-dimensional, which is a contradiction.
		
		Now suppose $\wt{k}_0^\theta=c\,k_0^\theta$. Then $c\ne0$ as $\wt{k_0^\theta}\neq 0$. Applying $S_\theta$ and using  \eqref{eq:s4} gives $c\,S_\theta k_0^\theta=-\theta(0)k_0^\theta$, whence
		$S_\theta k_0^\theta\in\spn\{k_0^\theta\}$, contradicting the first part.
\end{proof}

\begin{lemma}\label{lem:indep}
Let $\theta$ be a finite Blaschke product of degree $n$ and let $\lambda_{1},\dots,\lambda_{m}$ be distinct points of $\overline{\D}$ with $m\le n$. Then the reproducing kernels $k_{\lambda_{1}}^{\theta},\dots,k_{\lambda_{m}}^{\theta}$ are linearly independent in $\clk_\theta$.
\end{lemma}

\begin{proof}
Write $\theta(z)=\prod\limits_{j=1}^{n}\dfrac{a_j-z}{1-\overline{a_j}z}$, $a_j \in \D$ for all $j=1,\ldots, n$ and $q(z)=\prod\limits_{j=1}^{n}(1-\overline{a_j}z)$, so that $q$ has no zeros in $\overline{\D}$ and
		\[
		\clk_\theta=\Bigl\{\tfrac{p}{q}: p\ \text{a polynomial},\ \deg p\le n-1\Bigr\} .
		\]
For $f \in \clk_{\theta}$, the evaluation map $E(f)=(f(\lambda_1),\dots,f(\lambda_m))$ is onto $\C^{m}$. Indeed, given $(w_1,\dots,w_m) \in \C^m$, Lagrange interpolation produces a polynomial $p$ with $\deg p\le m-1\le n-1$ and $p(\lambda_i)=w_iq(\lambda_i)$, and then $f=p/q \in\clk_\theta$ satisfies $f(\lambda_i)=w_i$.

Now if $\sum\limits_{i=1}^m c_i\,k_{\lambda_i}^{\theta}=0$, then for $f \in \clk_\theta$,
		\begin{equation}\label{eq:ip}
			0=\left\langle f, \sum_{i=1}^m c_i k_{\lambda_i}^{\theta} \right\rangle=\sum_{i=1}^m \overline{c_i}\,f(\lambda_i).
		\end{equation}
As $E$ is surjective, \eqref{eq:ip} implies that the  functional $(w_1,\dots,w_m)\mapsto\sum_i\overline{c_i}w_i$ vanishes on $\C^{m}$, so $c_i=0$ for every $i$. This proves that $k_{\lambda_{1}}^{\theta},\dots,k_{\lambda_{m}}^{\theta}$ are linearly independent. 
\end{proof}

For inner functions that are not finite Blaschke products, the analogous statement is due to Bessonov \cite{RB}. In the following, $\Omega$ denotes the set of points of $\T$ at which $\theta$ has an ADC.
	
\begin{lemma}[{\cite[Proposition~2.8]{RB}}]\label{lem:bessonov}
		Suppose $\theta$ is not a finite Blaschke product, and let $\lambda_1,\dots,\lambda_m$ be distinct
		points of $\D\cup\Omega$. Then the reproducing kernels $k_{\lambda_1}^\theta,\dots,k_{\lambda_m}^\theta$
		are linearly independent in $\clk_\theta$.
\end{lemma}
	
Combining Lemma~\ref{lem:indep} and Lemma~\ref{lem:bessonov}, we conclude the following:
	
\begin{corollary}\label{cor:kindep}
Let $\theta$ be inner and let $\lambda_1,\dots,\lambda_m$ be distinct points of $\D\cup\Omega$ with $m\le\dim\clk_\theta$. Then $k_{\lambda_1}^\theta,\dots,k_{\lambda_m}^\theta$ are linearly independent.
\end{corollary}

	\newsection{The intersection \texorpdfstring{$\cli_\theta$}{I-theta}}\label{sec:upper}
	
If $\theta$ has an ADC at $\zeta\in\T$, then every function in $\clk_\theta$ has a non-tangential limit at $\zeta$ and the limit defines a
	bounded linear functional on $\clk_\theta$. Thus there is a corresponding kernel function  $k_\zeta^\theta \in \clk_\theta$, the conjugate kernel
	$\wt{k}_\zeta^\theta=C_\theta k_\zeta^\theta$ also lies in $\clk_\theta$, and
	$\norm{k_\zeta^\theta}^{2}=|\theta'(\zeta)|<\infty$; this is the Ahern--Clark theorem
	\cite{AhernClark} (see also \cite[Ch.~7]{GMR}). We have
	\[
	k_\zeta^\theta (z)=\frac{1-\overline{\theta(\zeta)}\theta}{1-\bar{\zeta}z}.
	\]
	We shall also use the boundary relation
	\begin{equation}\label{eq:bdryconj}
		\wt{k}_\zeta^\theta=\overline{\zeta}\,\theta(\zeta)\,k_\zeta^\theta,\qquad \zeta\in\T,
	\end{equation}
	which is immediate from the explicit formulas once $|\zeta|=|\theta(\zeta)|=1$.
	This section is devoted to obtain the complete characterization of $\cli_\theta$ if $\theta$ has ADC at $1$ and $-1$. More precisely, the object of this section is the following theorem.

\begin{theorem}\label{thm:main}
Let $\theta$ be an inner function with $\dim\clk_\theta\ge 2$ having an angular derivative in the sense of Carath\'eodory at $z=1$ and at $z=-1$. Then
		\[
		\cli_\theta=\spn\{k_1^\theta \otimes k_1^\theta,k_{-1}^\theta\otimes k_{-1}^\theta\}.
		\]
\end{theorem}

The following lemma can be found in \cite{Sarason} and will be used frequently in our results. 
	
	\begin{lemma}[{\cite[\S 2]{Sarason}}]\label{lem:bdry}
		Let $\zeta\in\T$ be a point at which $\theta$ has an ADC. Then
		\[
		S_\theta k_\zeta^\theta=\zeta\,(k_\zeta^\theta-k_0^\theta),\qquad
		S_\theta^* k_\zeta^\theta=\overline{\zeta}\,k_\zeta^\theta-\overline{\theta(\zeta)}\,\wt{k}_0^\theta.
		\]
	\end{lemma}

We seek to determine the conditions under which $k_\zeta^\theta \otimes k_\zeta^\theta$ is in $\clt_\theta$ as well as in $\clh_\theta$ whenever $\theta$ has an ADC at $\zeta$. The following result, in this context, can be found in \cite[Theorem 5.1]{Sarason}.  We state it here for ease of reference.
	
\begin{proposition}\label{prop:TTO}
		Let $\zeta\in\T$ be a point at which $\theta$ has an ADC. Then
		$k_\zeta^\theta\otimes k_\zeta^\theta\in\clt_\theta$, and
		\[
		D_T\bigl(k_\zeta^\theta\otimes k_\zeta^\theta\bigr)
		=k_\zeta^\theta\otimes k_0^\theta+k_0^\theta\otimes\bigl(k_\zeta^\theta-k_0^\theta\bigr)\in\clr_T.
		\]
\end{proposition}

To decide when $k_\zeta^\theta\otimes k_\zeta^\theta$ is a truncated Hankel operator as well, we require a criterion for inclusion in the Gu--Ma defect space $\clr_H$.
	
\begin{lemma}\label{lem:corner}
Let $P$ be the orthogonal projection of $\clk_\theta$ onto $(\C k_0^\theta)^{\perp}$ and $Q$ the orthogonal projection onto $(\C\wt{k}_0^\theta)^{\perp}$. Then, for $X\in\clb(\clk_\theta)$,
		\[
		X\in\clr_H\iff QXP=0 .
		\]
Consequently, for $w\in\clk_\theta$, $\ w\otimes w\in\clr_H$ if and only if $w\in\C k_0^\theta\cup\C\wt{k}_0^\theta$.
\end{lemma}

\begin{proof}
Note that $P,Q$ are projections with 
\[
\ker P=\C k_0^\theta, \quad \ker Q=\C\wt{k}_0^\theta,\quad \mbox{and} \quad I-P=\norm{k_0^\theta}^{-2}k_0^\theta\otimes k_0^\theta.
\]
Let $x, y \in \clk_\theta$ such that
		\[
		X=x\otimes k_0^\theta+\wt{k}_0^\theta\otimes y \in \clr_H.
		\]
Then
\[
		QXP=(Qx)\otimes(Pk_0^\theta)+(Q\wt{k}_0^\theta)\otimes(Py)=0, 
\]		
since $Pk_0^\theta=0$ and $Q\wt{k}_0^\theta=0$.

Conversely, suppose $QXP=0$. Then $XPf\in\ker Q=\C\wt{k}_0^\theta$ for every $f\in\clk_\theta$. Thus $XP$ has at most rank one with $\ran XP\subseteq \C\wt{k}_0^\theta$. Thus $XP=\wt{k}_0^\theta\otimes y$ for some $y\in\clk_\theta$. Also,
		\[X(I-P)=\norm{k_0^\theta}^{-2}(Xk_0^\theta)\otimes k_0^\theta=:x\otimes k_0^\theta.
		\]
Adding the two yields $X\in\clr_H$. For the last statement, $Q(w\otimes w)P=(Qw)\otimes(Pw)$ vanishes if and only if $Qw=0$ or $Pw=0$, that is, $w\in\C\wt{k}_0^\theta$ or $w\in\C k_0^\theta$.
\end{proof}

\begin{lemma}\label{lem:position}
Let $\theta$ be inner with $\dim\clk_\theta\ge2$ and let $\zeta\in\T$ be a point at which
$\theta$ has an ADC. Then $k_\zeta^\theta\notin\C k_0^\theta$ and $k_\zeta^\theta\notin\C\wt{k}_0^\theta$.
\end{lemma}
	
\begin{proof}
		Clearly, $k_\zeta^\theta\ne0$. Corollary~\ref{cor:kindep} shows that $k_0^\theta$ and $k_\zeta^\theta$ 		are linearly independent, so $k_\zeta^\theta\notin\C k_0^\theta$.
		
		Suppose next $k_\zeta^\theta=c\,\wt{k}_0^\theta$, necessarily with $c\ne0$. Applying $C_\theta$
		and using its antilinearity  gives $\wt{k}_\zeta^\theta=\overline{c}\,k_0^\theta$. On the
		other hand $\wt{k}_\zeta^\theta=\overline{\zeta}\theta(\zeta)k_\zeta^\theta$ by
		\eqref{eq:bdryconj}, with $\overline{\zeta}\theta(\zeta)$ unimodular. Hence
		$k_\zeta^\theta\in\C k_0^\theta$,  which is a contradiction. This completes the proof.
\end{proof}

\begin{proposition}\label{prop:Hdich}
Let $\theta$ be a non-constant inner function and let $\zeta\in\T$ be a point at which $\theta$ has an ADC. Then
\begin{equation}\label{eq:Hform}
			D_H\bigl(k_\zeta^\theta\otimes k_\zeta^\theta\bigr)
			=\bigl(1-\overline{\zeta}^{\,2}\bigr)\,k_\zeta^\theta\otimes k_\zeta^\theta
			+\overline{\zeta}^{\,2}\,k_\zeta^\theta\otimes k_0^\theta
			+\overline{\zeta}\,\overline{\theta(\zeta)}\,\bigl(\wt{k}_0^\theta\otimes(k_\zeta^\theta-k_0^\theta)\bigr).
\end{equation}
Consequently,
\begin{enumerate}
			\item if $\zeta\in\{1,-1\}$, then $k_\zeta^\theta\otimes k_\zeta^\theta\in\clh_\theta$;
			\item if $\dim\clk_\theta\ge2$, then
			$k_\zeta^\theta\otimes k_\zeta^\theta\in\clh_\theta$ if and only if $\zeta\in\{1,-1\}$.
\end{enumerate}
\end{proposition}

\begin{proof}
The asserted identity follows by straightforward calculation and invoking Lemma \ref{lem:bdry}. Since the last two terms of \eqref{eq:Hform} are in $\clr_H$ and $\clr_H$ is a linear subspace,
		\begin{equation}\label{eq:reduce}
			D_H\bigl(k_\zeta^\theta\otimes k_\zeta^\theta\bigr)\in\clr_H
			\iff
			\bigl(1-\overline{\zeta}^{\,2}\bigr)\,k_\zeta^\theta\otimes k_\zeta^\theta\in\clr_H .
		\end{equation}
		
		(i) If $\zeta=\pm1$ then $\overline{\zeta}^{\,2}=1$, so  \eqref{eq:reduce} holds trivially and hence $k_\zeta^\theta\otimes k_\zeta^\theta\in\clh_\theta$.
		
		(ii) If $\zeta\ne\pm1$ then $1-\overline{\zeta}^{\,2}\ne0$, and since $\clr_H$ is a linear
		subspace, right-hand condition of \eqref{eq:reduce} reduces to $k_\zeta^\theta\otimes k_\zeta^\theta\in\clr_H$. By
		Lemma~\ref{lem:corner}, this would force $k_\zeta^\theta\in\C k_0^\theta\cup\C\wt{k}_0^\theta$,
		which contradicts Lemma~\ref{lem:position}. Hence
		$k_\zeta^\theta\otimes k_\zeta^\theta\notin\clh_\theta$.
	\end{proof}

For the rest of the article, for $\zeta \in \T$ a point at which $\theta$ has an ADC, we denote $P_\zeta=k_\zeta^\theta \otimes k_\zeta^\theta$. Combining Proposition \ref{prop:TTO} and Proposition \ref{prop:Hdich}, we get the following result.
\begin{proposition}\label{prop:gens}
		Let $\theta$ be a non-constant inner function with $\dim\clk_\theta\ge2$ having an ADC at $z=1$
		and $z=-1$. Then $P_1$ and $P_{-1}$ are linearly independent elements of $\cli_\theta$; in
		particular $\cli_\theta\supseteq\spn\{P_1,P_{-1}\}$ and $\dim\cli_\theta\ge2$.
		If $\theta$ has an
		ADC at exactly one of $\pm1$, say $\zeta_0$, then $P_{\zeta_0}\in\cli_\theta$ and $P_{\zeta_0}\ne0$.
\end{proposition}
	
	\begin{proof}
		Proposition~\ref{prop:TTO} and Proposition~\ref{prop:Hdich}(i)  yield $P_\zeta\in\clt_\theta \cap \clh_\theta=\cli_\theta$ for $\zeta=\pm1$
		whenever the
		corresponding ADCs hold. Furthermore, $P_\zeta\ne0$. Now assume that there exist $\alpha,\beta \in \C$ such that
		$\alpha P_1+\beta P_{-1}=0$. When both ADCs hold,
		$k_1^\theta,k_{-1}^\theta$ are linearly independent, so there is $h\in\clk_\theta$
		with $\ip{h}{k_{-1}^\theta}=0$ and $\ip{h}{k_1^\theta}=1$.  Then
		\[
		0=(\alpha P_1+\beta P_{-1})h=
		\alpha \left\langle h,k_{1}^\theta\right\rangle k_1^\theta +\beta \left\langle h, k_{-1}^\theta \right\rangle k_{-1}^\theta= \alpha k_1^\theta.
		\]
		Thus $\alpha=0$ which, in turn, implies $\beta=0$. This finishes the proof.
	\end{proof}

Now to prove that the inclusion in Proposition \ref{prop:gens} is, in fact, equality, we need some necessary framework. Write
	\[
	\clv_1=\bigl\{h \in \clk_\theta: (S_\theta^*-S_\theta)h\in \spn\bigl\{k_0^\theta,\wt{k}_0^\theta\bigr\}\bigr\},
	\qquad
	\clv_2=\bigl\{h \in \clk_\theta: (S_\theta^2-I)h\in \spn\bigl\{k_0^\theta,S_\theta{k_0^\theta}\bigr\}\bigr\}.
	\]
	The role of $\clv_1$ is that the range of every $C\in\cli_\theta$ will be shown to lie inside it, and $\clv_2$ is the preimage of $ \spn\{k_{0}^\theta, S_\theta k_0^{\theta} \}$ under the injective operator $I-S_\theta^2$.  
The following lemma is elementary yet plays a significant role in proving our main result.

\begin{lemma}\label{lem:V1V2}
For every non-constant inner function $\theta$, one has $\clv_1=\clv_2$.
\end{lemma}

\begin{proof}
Let $h\in \clv_1$. Suppose $(S_\theta^{*}-S_\theta)h=a\,k_0^\theta+b\,\wt{k}_0^\theta$. Applying $S_\theta$ and using \eqref{eq:s1} and \eqref{eq:s4}  give
		\[
		\bigl(I-k_0^\theta\otimes k_0^\theta-S_\theta^{2}\bigr)h
		=a\,S_\theta k_0^\theta-b\,\theta(0)\,k_0^\theta,
		\]
		which implies
		\begin{equation}\label{eq:forward}
			(I-S_\theta^{2})h=a\,S_\theta k_0^\theta+\bigl(\ip{h}{k_0^\theta}-b\,\theta(0)\bigr)k_0^\theta
			\in\spn\{k_0^\theta,S_\theta k_0^\theta\},
		\end{equation}
		so $h\in \clv_2$.
		
Conversely, let $h\in \clv_2$, say $(S_\theta^{2}-I)h=c\,k_0^\theta+d\,S_\theta k_0^\theta$. Applying
		$S_\theta^{*}$ and using \eqref{eq:s2}, \eqref{eq:s3} yield
		\[
		\bigl(I-\wt{k}_0^\theta\otimes\wt{k}_0^\theta\bigr)S_\theta h-S_\theta^{*}h =-c\,\overline{\theta(0)}\,\wt{k}_0^\theta
		+d\bigl(k_0^\theta-\ip{k_0^\theta}{\wt{k}_0^\theta}\wt{k}_0^\theta\bigr).
		\]
		Equating and solving for $\bigl(S_\theta^{*}-S_\theta\bigr)h$,
		\begin{equation}\label{eq:backward}
			\bigl(S_\theta^{*}-S_\theta\bigr)h
			=-d\,k_0^\theta
			+\Bigl(c\,\overline{\theta(0)}+d\,\ip{k_0^\theta}{\wt{k}_0^\theta}
			-\ip{S_\theta h}{\wt{k}_0^\theta}\Bigr)\wt{k}_0^\theta
			\in\spn\{k_0^\theta,\wt{k}_0^\theta\},
		\end{equation}
		so $h\in \clv_1$.
	\end{proof}

\begin{remark}\label{rem:V1}
For every non-constant inner function  $\theta$, the operators  $I-S_\theta^2$ and $I\pm S_\theta $ are injective.  Indeed, if  $S_\theta^{2}h=h$, then $S_\theta^{2m}h=h$ for every $m\ge1$; since $S_\theta$ is of class $C_{00}$, $\|S_\theta^{2m}h\|\to0$, so $h=0$. Similar proof will work for $I\pm S_\theta$ if $S_\theta h=\pm h$.	 
\end{remark}

Now, we can have an explicit representation of $\clv_1$ (hence, $\clv_2$) if $\theta$ has ADC at $1$ and $-1$. 
	
	\begin{lemma}\label{lem:V1}
		Let $\theta$ be an inner function with $\dim \clk_\theta\ge 2$ (possibly infinite) having an ADC
		at $1$ and at $-1$. Then $\clv_1=\clv_2=\spn\{k_1^\theta,k_{-1}^\theta\}$, so that $\dim \clv_1=2$.
	\end{lemma}
	
	\begin{proof}
		By Lemma~\ref{lem:cyclic}, $\{k_0^\theta,S_\theta k_0^\theta\}$ is linearly
		independent; write
		\[
		\clw=\spn\{k_0^\theta,S_\theta k_0^\theta\}=\spn\{k_0^\theta+S_\theta k_0^\theta, k_0^\theta-S_\theta k_0^\theta\},
		\]
		a two-dimensional space. Let $L=I-S_\theta^2$, then routine calculations yield
		\[
		L\,k_{1}^\theta=k_0^\theta+S_\theta k_0^\theta \qquad \text{and} \qquad
		L\,k_{-1}^\theta=k_0^\theta-S_\theta k_0^\theta.
		\] 
		Thus $k_1^\theta, k_{-1}^\theta \in \clv_2=\{h:L\,h\in \clw\}$. Since $L$ is injective, $\clv_2=\spn\{k_1^\theta,k_{-1}^\theta\}$, and $\clv_1=\clv_2$ by Lemma \ref{lem:V1V2}. The two kernels are independent by Corollary~\ref{cor:kindep}, so
		$\dim \clv_2=2$.
	\end{proof}
	
	\begin{remark}\label{rem:dim2}
		It is worth mentioning that if $\dim \clk_{\theta}=2$ so that  $\dim\clb(\clk_\theta)=4$, then, choosing $\zeta\in\T\setminus\{1,-1\}$, 	Propositions~\ref{prop:TTO} and \ref{prop:Hdich}(ii) give 	$k_\zeta^\theta\otimes k_\zeta^\theta\in\clt_\theta\setminus\clh_\theta$ so
		$\clt_\theta\ne\clh_\theta$. Both the classes have dimension $2 \dim \clk_{\theta}-1=3$ (by Section \ref{sec:prelim}) and two distinct $3$-dimensional subspaces of the $4$-dimensional space $\clb(\clk_\theta)$ together span it,  so
		$\dim(\clt_\theta+\clh_\theta)=4$. Thus $\dim\cli_\theta=3+3-4=2$, and
		Proposition~\ref{prop:gens} identifies $\cli_\theta$ as $\spn\{P_1,P_{-1}\}$.
	\end{remark}
	
	\begin{example}
		For $\theta=z^2$, $\clk_\theta$ has an orthonormal basis $\{1,z\}$ and corresponding to this basis, $\clt_\theta$ and $\clh_\theta$ are precisely the classes of all $2 \times 2$ Toeplitz and Hankel matrices, respectively. By the above remark, $\clt_\theta +\clh_\theta=\clb(\clk_\theta)=M_{2 \times 2}(\C)$. Indeed, for $a,b,c,d \in \C$, we can always write
		\[
		\begin{bmatrix}
			a & b\\
			c & d
		\end{bmatrix}
		=
		\begin{bmatrix}
			a & b\\
			c & a
		\end{bmatrix}
		+
		\begin{bmatrix}
			0 & 0\\
			0 & d-a
		\end{bmatrix},
		\] 
		where the first matrix of the sum is clearly in $\clt_\theta$ and second one is in $\clh_\theta$.
	\end{example}
	
	By Remark~\ref{rem:dim2}, we are left to prove Theorem~\ref{thm:main} for $\dim \clk_\theta \geq 3$. The following lemma will be useful in our main proof and is also of independent interest.
	
	\begin{lemma}\label{lem:defect}
		Let $\theta$ be an inner function with $\dim\clk_\theta\ge3$ and let $C\in\cli_\theta$.
		We may write
		\begin{equation}\label{eq:twodefects}
			D_T(C)=\psi\otimes k_0^\theta+k_0^\theta\otimes\chi,
			\qquad
			D_H(C)=f\otimes k_0^\theta+\wt{k}_0^\theta\otimes g
		\end{equation}
		with $\psi,\chi,f,g\in \clk_\theta$. Then
		\[
		\psi-f\in\spn\{k_0^\theta,\wt{k}_0^\theta\}
		\qquad\text{and}\qquad
		S_\theta^*\psi-S_\theta f\in\spn\{k_0^\theta,\wt{k}_0^\theta\}.
		\]
	\end{lemma}
	
	\begin{proof}
		We expand $D_HD_T(C)$ in two ways. First, applying $D_H$ to $D_T(C)$ and using
		relations  \eqref{eq:disp} and 	\eqref{eq:s3}, we obtain
		\begin{equation}\label{eq:way1}
			D_H D_T(C)=\psi\otimes k_0^\theta-\bigl(S_\theta^{*}\psi\bigr)\otimes\bigl(S_\theta k_0^\theta\bigr)
			+k_0^\theta\otimes\chi
			+\overline{\theta(0)}\,\bigl(\wt{k}_0^\theta\otimes S_\theta\chi\bigr).
		\end{equation}
		
		On the other hand, we can write $D_HD_T=D_TD_H+[D_H,D_T]$, using
		\eqref{eq:s4} and  \eqref{eq:comm},
		\begin{equation}\label{eq:way2}
			D_H D_T(C)=f\otimes k_0^\theta-\bigl(S_\theta f\bigr)\otimes\bigl(S_\theta k_0^\theta\bigr)
			+\wt{k}_0^\theta\otimes g+\theta(0)\,\bigl(k_0^\theta\otimes S_\theta g\bigr)
			+k_0^\theta\otimes\bigl(S_\theta^{2}C^{*}k_0^\theta\bigr)
			-\wt{k}_0^\theta\otimes\bigl(S_\theta^{2}C^{*}\wt{k}_0^\theta\bigr).
		\end{equation}
		
		Equating \eqref{eq:way1} and \eqref{eq:way2}, we get
		\begin{equation}\label{eq:star}
			(\psi-f)\otimes k_0^\theta
			-\bigl(S_\theta^{*}\psi-S_\theta f\bigr)\otimes\bigl(S_\theta k_0^\theta\bigr)
			+k_0^\theta\otimes \phi_1+\wt{k}_0^\theta\otimes \phi_2=0,
		\end{equation}
		where $\phi_1=\chi-\overline{\theta(0)}S_\theta g-S_\theta^{2}C^{*}k_0^\theta$ and
		$\phi_2=\theta(0)S_\theta\chi-g+S_\theta^{2}C^{*}\wt{k}_0^\theta$.
		
		Applying \eqref{eq:star} to $h\in\clk_\theta$ gives
		\[
		\ip{h}{k_0^\theta}(\psi-f)
		-\ip{h}{S_\theta k_0^\theta}\bigl(S_\theta^{*}\psi-S_\theta f\bigr)
		+\ip{h}{\phi_1}\,k_0^\theta+\ip{h}{\phi_2}\,\wt{k}_0^\theta=0 .
		\]
		Since $\dim\clk_\theta\ge3$, we may choose $0 \neq h \in \spn\{k_0^\theta,S_\theta k_0^\theta\}^\perp$, then the first two terms
		vanish. Therefore, $\ip{h}{\phi_1}k_0^\theta+\ip{h}{\phi_2}\wt{k}_0^\theta=0$ which, in turn, implies
		$\ip{h}{\phi_1}=\ip{h}{\phi_2}=0$ because $k_0^\theta,\wt{k}_0^\theta$ are linearly independent. As this holds for every such $h$, we conclude $\phi_1,\phi_2\in\spn\{k_0^\theta,S_\theta k_0^\theta\}$.
		
		Write
		$\phi_1=a_{1}k_0^\theta+a_{2}S_\theta k_0^\theta$ and $\phi_2=b_{1}k_0^\theta+b_{2}S_\theta k_0^\theta$. Then
		\[
		k_0^\theta\otimes \phi_1=\overline{a_{1}}\bigl(k_0^\theta\otimes k_0^\theta\bigr)
		+\overline{a_{2}}\bigl(k_0^\theta\otimes S_\theta k_0^\theta\bigr),
		\]
		and similarly for
		$\wt{k}_0^\theta\otimes \phi_2$. Substituting into \eqref{eq:star}, we obtain
		\[
		\Bigl[(\psi-f)+\overline{a_{1}}k_0^\theta+\overline{b_{1}}\wt{k}_0^\theta\Bigr]\otimes k_0^\theta
		+\Bigl[-\bigl(S_\theta^{*}\psi-S_\theta f\bigr)+\overline{a_{2}}k_0^\theta
		+\overline{b_{2}}\wt{k}_0^\theta\Bigr]\otimes\bigl(S_\theta k_0^\theta\bigr)=0 .
		\]
		If $p\otimes q_{1}+r\otimes q_{2}=0$ with $q_1,q_2$ linearly independent then $p=r=0$. Indeed, we can choose nonzero	$h$ with $\ip{h}{q_1}=1$, $\ip{h}{q_2}=0$ and apply the relation to $h$, then symmetrically.
		Since $k_0^\theta$ and $S_\theta k_0^\theta$ are linearly independent, therefore  $(\psi-f)+\overline{a_{1}}k_0^\theta+\overline{b_{1}}\wt{k}_0^\theta=0$ and $-\bigl(S_\theta^{*}\psi-S_\theta f\bigr)+\overline{a_{2}}k_0^\theta
		+\overline{b_{2}}\wt{k}_0^\theta =0 $,
		which is our assertion.
	\end{proof}
	
We are just one step away from proving our main result, Theorem \ref{thm:main}. Before that, we define the \emph{Crofoot transform} \cite{Crofoot}, which will be helpful for our proof. One can refer to \cite[\S13]{Sarason} for the following.

\begin{remark}\label{rem:crofoot}
For $w \in \D$ put $\theta_w=\dfrac{\theta-w}{1-\overline{w}\theta}$, which is again inner, and define the \emph{Crofoot transform} $J_w: \clk_\theta \rightarrow \clk_{\theta_w}$ by
\[
J_w f=\frac{\sqrt{1-|w|^{2}}}{1-\overline{w}\theta}\,f, \quad (f \in \clk_\theta) .
\]
Then $J_w$ is a unitary operator from $\clk_\theta$ onto $\clk_{\theta_w}$, hence, $\dim\clk_{\theta_w}=\dim\clk_\theta$.
		Moreover, the transform preserves both classes,
		\[
		J_w \clt_\theta J_w^{-1}=\clt_{\theta_w}
		\qquad\text{and}\qquad
		J_w \clh_\theta J_w^{-1}=\clh_{\theta_w},
		\]
		the first statement being due to Sarason \cite[\S13]{Sarason} and the second to Gu and Ma
		\cite[Theorem 9.3]{GuMa}. Since conjugation by a bijection preserves intersections, it follows
		that $J_w \cli_{\theta}J_w^{-1}=\cli_{\theta_w}$. In addition, Crofoot transform $J_w$ satisfies the following identities (cf. \cite[Lemma 13.2]{Sarason}):
		
		\begin{equation}\label{eq:crofoot}
			J_w k_{\lambda}^\theta=\frac{1-w\overline{\theta(\lambda)}}{\sqrt{1-|w|^{2}}}\,k_\lambda^{\theta_w},
			\qquad
			J_w \wt{k}_{\lambda}^\theta=\frac{1-\overline{w}\theta(\lambda)}{\sqrt{1-|w|^{2}}}\,\wt{k}_\lambda^{\theta_w} \qquad (\lambda \in \D).
		\end{equation}
\end{remark}

\begin{lemma}\label{lem:transfer}
Let $\theta$ be inner and $w\in\D$. Then $\theta$ has an ADC at $\zeta$ where $\zeta\in\T$ if and only if $\theta_w$ does.
\end{lemma}
	
\begin{proof}
Let $\phi_w(\lambda)=(\lambda-w)(1-\overline{w}\lambda)^{-1}$, a M\"obius automorphism
of $\D$ with inverse $\phi_{-w}$. Its only pole is at $\lambda=1/\overline{w} \notin \overline{\D}$, so $\phi_w$ is holomorphic on a neighbourhood of $\overline{\D}$ and maps $\T$
		onto $\T$ with non-vanishing derivative. Since $\theta_w=\phi_w\circ\theta$ and
		$\theta=\phi_{-w}\circ\theta_w$, it suffices to prove one implication. If $\theta$ has an ADC at
		$\zeta$, then by the Julia--Carath\'eodory theorem (Lemma~\ref{lem:Julia}), a non-tangential approach to $\zeta$ yields
		\[
		\frac{1-|\theta_w(z)|}{1-|z|}
		=\frac{1-|\phi_w(\theta(z))|}{1-|\theta(z)|}\cdot\frac{1-|\theta(z)|}{1-|z|}
		\longrightarrow|\phi_w'(\theta(\zeta))|\cdot|\theta'(\zeta)|<\infty,
		\]
		the first factor tending to $|\phi_w'(\theta(\zeta))|\in(0,\infty)$ because $\phi_w$ is conformal
		at the unimodular point $\theta(\zeta)$. Recall that the convergence here is non-tangential. Again by Lemma~\ref{lem:Julia}, $\theta_w$ has an
		ADC at $\zeta$. The converse can be proved by the same argument applied to $\phi_{-w}$ and $\theta_w$.
	\end{proof}

\begin{lemma}\label{lem:boundary}
		The identities in (\ref{eq:crofoot}) also hold whenever $\zeta\in\T$ is a point at which $\theta$ has an ADC, the
		kernels $k_\zeta^\theta$ and $k_\zeta^{\theta_w}$ then being elements of $\clk_\theta$
		and $\clk_{\theta_w}$.
\end{lemma}

\begin{proof}
Let  $\zeta\in\T$ be a point at which $\theta$ has an ADC. By Lemma~\ref{lem:transfer}, $\theta_w$	also has an ADC at $\zeta$, hence $k_\zeta^\theta$, $k_\zeta^{\theta_w}$, $\wt{k}_\zeta^{\theta}$ and $\wt{k}_\zeta^{\theta_w}$ lie in the respective model spaces and are the non-tangential norm-limits of $k_\lambda^\theta$,  $k_\lambda^{\theta_w}$, $\wt{k}_\lambda^{\theta}$ and $\wt{k}_\lambda^{\theta_w}$ respectively as $\lambda \to \zeta$. Passing to this limit in the identities of (\ref{eq:crofoot}) for $\lambda\in\D$, and using $\theta(\lambda)\to\theta(\zeta)$ non-tangentially, proves the assertion.
	\end{proof}

Now we are ready to give the proof of the main theorem.

\begin{proof}[Proof of Theorem~\ref{thm:main}]
		By Remark \ref{rem:dim2}, we may assume that $\dim \clk_{\theta} \geq 3$. We will first prove the result when  $\theta(0)\neq 0$.
		Let $C\in\cli_\theta$ such that
		\begin{equation}\label{eq:twodefects1}
			D_T(C)=\psi\otimes k_0^\theta+k_0^\theta\otimes\chi
			\qquad \text{and} \qquad
			D_H(C)=f\otimes k_0^\theta+\wt{k}_0^\theta\otimes g
		\end{equation}
		with $\psi,\chi,f,g\in \clk_\theta$
		and note that
		$\clv_1=\spn\{k_1^\theta,k_{-1}^\theta\}$ by Lemma~\ref{lem:V1}. We show that  $\ran C\subseteq \clv_1$ and $\ran C^* \subseteq \clv_1$. Subtracting the two identities in \eqref{eq:twodefects1} yields
		\[
		\bigl(S_\theta^{*}-S_\theta\bigr)C\,S_\theta^{*}=D_T(C)-D_H(C)
		=(\psi-f)\otimes k_0^\theta+k_0^\theta\otimes\chi-\wt{k}_0^\theta\otimes g .
		\]
		By Lemma~\ref{lem:defect}, we have $\psi-f\in\spn\bigl\{k_0^\theta,\wt{k}_0^\theta\bigr\}$, so the range of
		the right-hand side is contained in $\spn\{k_0^\theta,\wt{k}_0^\theta\}$. Since $\theta(0)\ne0$,
		$S_\theta^{*}$ is invertible, in particular, surjective, so
		\[
		\ran\bigl((S_\theta^{*}-S_\theta)CS_\theta^{*}\bigr)=\bigl(S_\theta^{*}-S_\theta\bigr)(\ran C).
		\]
		
		Hence $\bigl(S_\theta^{*}-S_\theta\bigr)(\ran C)\subseteq\spn\{k_0^\theta,\wt{k}_0^\theta\}$,
		that is, $\ran C\subseteq \clv_1$. Again  $C\in\clt_\theta$  is $C_\theta$-symmetric, that is,
		$C^{*}=C_\theta CC_\theta$. Then  $\ran C^{*}=C_\theta(\ran C)\subseteq C_\theta(\clv_1)$.
		
		Since $C_\theta S_\theta C_\theta=S_\theta^{*}$, we have $C_\theta(S_\theta^{*}-S_\theta)C_\theta=-(S_\theta^{*}-S_\theta)$, and $C_\theta$ maps $\spn \{ k_0^\theta, \wt{k}_0^\theta\}$ onto itself because $C_\theta k_0^\theta=\wt{k}_0^\theta$;  and it follows  $C_\theta(\clv_1)=\clv_1$.
		Hence $\ran C^* \subseteq \clv_1$.
		
		Therefore $\ran C \subseteq \clv_1$ and $\ran C^*  \subseteq \clv_1$, so $C=P_{\clv_1}CP_{\clv_1}$, where $P_{\clv_1}$ is the orthogonal projection onto $\clv_1$. Hence $C$ lies in the space
		\[
		\{X \in \clb(\clk_\theta):\ran X\subseteq \clv_1,\ \ran X^{*}\subseteq \clv_1\},
		\]
		which is isomorphic to $\clb(\clv_1)$ and  thus  has dimension
		$(\dim \clv_1)^{2}=4$. Note that the four operators $k_i^\theta \otimes k_j^\theta$ with $i,j \in \{-1,1\}$ belong to it and are linearly independent by Corollary~\ref{cor:kindep}, hence form a basis. Therefore, there exist $c_{i,j} \in \C$, $i,j \in \{-1,1\}$, such that
		\[
		C=c_{1,1}\,k_1^\theta\otimes k_1^\theta+c_{1,-1}\,k_1^\theta\otimes k_{-1}^\theta
		+c_{-1,1}\,k_{-1}^\theta\otimes k_1^\theta+c_{-1,-1}\,k_{-1}^\theta\otimes k_{-1}^\theta .
		\]
		
		It remains to show that the cross terms vanish. Proposition~\ref{prop:TTO} gives $D_T\bigl(k_{i}^{\theta}\otimes k_{i}^{\theta}\bigr)\in\clr_T$. Writing $\varepsilon_{\pm1}=\pm1$, Lemma~\ref{lem:bdry} 	at $\zeta=\pm1$ reads $S_\theta k_{i}^{\theta}=\varepsilon_i(k_{i}^{\theta}-k_0^\theta)$. For
		$i\ne j$, we have $\varepsilon_i\varepsilon_j=-1$, so
		\[
		D_T\bigl(k_{i}^{\theta}\otimes k_{j}^{\theta}\bigr)
		=k_{i}^{\theta}\otimes k_{j}^{\theta}
		+\bigl(k_{i}^{\theta}-k_0^\theta\bigr)\otimes\bigl(k_{j}^{\theta}-k_0^\theta\bigr)
		\equiv 2\,k_{i}^{\theta}\otimes k_{j}^{\theta} \qquad ( \text{mod} \, \clr_T).
		\]
		Since $C\in\clt_\theta$, we have $D_T(C)\in\clr_T$, so there exist $u,v\in\clk_\theta$ with
		\begin{equation}\label{eq:cross}
			2c_{1,-1}\,k_1^\theta\otimes k_{-1}^\theta+2c_{-1,1}\,k_{-1}^\theta\otimes k_1^\theta
			=u\otimes k_0^\theta+k_0^\theta\otimes v .
		\end{equation}
		By Corollary~\ref{cor:kindep},  $k_0^\theta,k_1^\theta,k_{-1}^\theta$ are linearly 	independent and since $\dim\clk_\theta\ge3$, we may choose $h\in \clk_\theta$ with $\ip{h}{k_0^\theta}=\ip{h}{k_1^\theta}=0$ and
		$\ip{h}{k_{-1}^\theta}=1$. Applying \eqref{eq:cross} to $h$ gives
		\[
		2c_{1,-1}k_1^\theta=\ip{h}{v}k_0^\theta\in\C k_0^\theta,
		\]
		and $k_1^\theta\notin\C k_0^\theta$ by
		Lemma~\ref{lem:position}, so $c_{1,-1}=0$. Similarly, $c_{-1,1}=0$.
		
		Hence $C=c_{1,1}P_1+c_{-1,-1}P_{-1}$, that is,  $\cli_\theta \subseteq \spn\{P_1,P_{-1}\}$. The 	reverse inclusion is already established in  Proposition~\ref{prop:gens}.\vspace{1mm}
		
		Now, for  $\theta(0)=0$,
		fix	$w\in\D\setminus\{0\}$. Then 
		$\dim\clk_{\theta_w}\ge3$, $\theta_w$ has an ADC at $\pm1$ by Lemma \ref{lem:transfer}, and $\theta_w(0)=-w\ne0$, so previous case  gives
		\[
		\cli_{\theta_w}=\spn\{k_1^{\theta_w}\otimes k_1^{\theta_w},k_{-1}^{\theta_w}\otimes k_{-1}^{\theta_w}\}.
		\]
		Since $\cli_\theta=J_w^{-1}\cli_{\theta_w}J_w$,  $\dim\cli_\theta=2$
		and since $\theta$ has an ADC at $\pm1$, Lemma \ref{lem:boundary} gives
		\[
		J_w^{-1}\bigl(k_{\pm1}^{\theta_w}\otimes k_{\pm1}^{\theta_w}\bigr)J_w
		\in\C\,P_{\pm1},
		\]
		and therefore $\cli_\theta=\spn\{P_1,P_{-1}\}$.
	\end{proof}

	\newsection{The complete determination of \texorpdfstring{$\cli_\theta$}{theta}}\label{sec:tri}
	
Theorem~\ref{thm:main} reveals the structure of $\cli_\theta$ for inner functions $\theta$ with an ADC at both $\pm1$. We now remove that restriction and determine $\cli_\theta$ for any arbitrary inner function. The main steps involve a criterion, in terms of the range of  $I-S_\theta^{2}$, for the presence of a generator at $\zeta\in\{1,-1\}$ and the determination of the set $\clv_2$ (or $\clv_1$). This criterion holds if and only if $\theta$ has an ADC at $\zeta$. Recall from Remark~\ref{rem:V1} that $I-S_\theta^2$ and $I \pm S_\theta$ are injective for every non-constant inner $\theta$. Before going to the main theorem, we derive some general results which will be instrumental for its proof, in particular, for the determination of $\clv_2$.
	
	\begin{lemma}\label{lem:reduce}
		For every non-constant inner $\theta$,
		\[
		k_0^\theta+S_\theta k_0^\theta\in\ran\bigl(I-S_\theta^{2}\bigr)\iff k_0^\theta\in\ran(I-S_\theta),
		\]
		and symmetrically, $k_0^\theta-S_\theta k_0^\theta\in\ran\bigl(I-S_\theta^{2}\bigr)\iff k_0^\theta\in\ran(I+S_\theta)$.
	\end{lemma}
	
\begin{proof}
Writing $I-S_\theta^2=(I+S_\theta)(I-S_\theta)$ and injectivity of $I+S_\theta$ exhibit the first statement. The latter one can be established by interchanging the role of $I-S_\theta$ and $I+S_\theta$.
\end{proof}

	\begin{lemma}\label{lem:solve}
		Let $\theta$ be a non-constant inner function. Then $k_0^\theta \in \ran (I-S_\theta)$ if and only if $\theta$ has an ADC at $1$. Symmetrically, $k_0^\theta \in \ran (I+S_\theta)$ if and only if $\theta$ has an ADC at $-1$.
	\end{lemma}
	
	\begin{proof}
		Let $\zeta \in\{1,-1\}$. For $g\in H^2$, the relation  $P_{\theta H^2}=M_{\theta}P_+ M_{\bar{\theta}}$ yields  $P_{\theta H^2}g=\theta\,P_+(\overline{\theta}g)$. Hence, for $h \in\clk_\theta$, 	$ S_\theta h=P_{\clk_\theta}(zh)=zh-\theta\,P_+(\overline{\theta}z h).$ Note that $P_+(\overline{\theta}h)=0$, i.e., $\overline{\theta}h$ has only strictly negative	Fourier powers, so	$\gamma:= \zeta P_+(\overline{\theta}z h)$ is a {constant}. Therefore
		\[
		(I-\zeta S_\theta)h=(1-\zeta z)h+\gamma\,\theta.
		\]		
Thus if $k_0^\theta \in \ran (I-\zeta S_\theta)$, then there exists $h \in \clk_\theta$ such that
		
		\[
		(1-\zeta z)h +\gamma \theta=k_0^\theta=1-\overline{\theta(0)}\theta.
		\]
Hence $h=\dfrac{1-\beta\theta}{1-\zeta z}$ with $\beta=\overline{\theta(0)}+\gamma$. We claim that $|\beta|=1$. Indeed, otherwise  $|1-\beta\theta|\ge\bigl||\beta|-1\bigr|>0$ a.e.\ on $\T$, and since
		\[
		\frac1{2\pi}\int_0^{2\pi}|1-{\zeta}e^{it}|^{-2}\,dt=\infty,
		\]
		this would give $\norm{h}_{H^2}=\infty$. Thus $k_0^\theta \in \ran (I-\zeta S_\theta)$ implies that there exist $\beta\in \T$ such that $\dfrac{1-\beta\theta}{1-\zeta z} \in \clk_\theta$ which, by Lemma~\ref{lem:Julia}, is equivalent of the fact that $\theta$ has an ADC at $\zeta$. Converse is easy to prove once we observe that if $\theta$ has ADC at $\zeta$, then 
		$k_\zeta^\theta \in \clk_\theta$ and $(I-\zeta S_\theta)k_{\zeta}^\theta =k_0^\theta$.
\end{proof}

An immediate consequence of Lemmas~\ref{lem:reduce} and \ref{lem:solve} is the following result:
	
\begin{theorem}\label{thm:range}
Let $\theta$ be a non-constant inner function, and let $\zeta\in\{1,-1\}$. Then
		\[
		k_0^\theta+\zeta\,S_\theta k_0^\theta\in\ran(I-S_\theta^{2})
		\]
		if and only if $\theta$ has an ADC at $\zeta$. In that case $k_\zeta^\theta$ in $\clk_\theta$, so that $P_\zeta$ is defined, and $k_\zeta^\theta$ satisfies the preimage relation  
		\[
		(I-S_\theta^{2})k_\zeta^\theta=k_0^\theta+\zeta\,S_\theta k_0^\theta.
		\]
\end{theorem}

Having established the preceding results, we are now in the position to give the explicit characterization of $\cli_\theta$ for $\dim \clk_\theta \geq 3$. The following theorem provides a complete description of this class.

\begin{theorem}\label{thm:dim}
Let $\theta$ be an inner function with $\dim\clk_\theta\ge3$. Then
\begin{align*}
			\cli_\theta&=\spn\{\,P_\zeta:\zeta\in\{1,-1\},\ \theta\ \text{has an ADC at}\ \zeta\,\}, \\
			\qquad
			\dim\cli_\theta&=\#\{\zeta\in\{1,-1\}:\theta\ \text{has an ADC at}\ \zeta\}\in\{0,1,2\}.
\end{align*}
\end{theorem}

\begin{proof}
		Set $w_\pm=k_0^\theta\pm S_\theta k_0^\theta=(I\pm S_\theta)k_0^\theta$, so that $\clw=\spn\{w_+,w_-\}$
		is two-dimensional by Lemma~\ref{lem:cyclic}. Write $d$ for the number of points $\zeta\in\{1,-1\}$
		at which $\theta$ has an ADC. We may assume $\theta(0)\ne0$ (otherwise we replace $\theta$ by a Crofoot
		transform $\theta_w$ with $w\ne0$, which preserves $d$ (Lemma~\ref{lem:transfer}) and
		$\dim\cli_\theta$ (Remark~\ref{rem:crofoot}) and carries each $P_\zeta$ to a multiple of
		$P_\zeta^{\theta_w}$). The first part of the proof of Theorem~\ref{thm:main}, which uses only $\dim\clk_\theta\ge3$ and $\theta(0)\ne0$, shows that if $C \in \cli_\theta$ then $\ran C \subseteq \clv_1$ and $\ran C^* \subseteq \clv_1$.
		Also, $\clv_1=\clv_2$ which is exactly the preimage of $\clw\cap\ran(I-S_\theta^{2})$ under $I-S_\theta^2$.  We claim that
		\[
		\clw\cap\ran(I-S_\theta^{2})=\spn\{\,w_\zeta:\zeta\in\{1,-1\},\ \theta\ \text{has an ADC at}\ \zeta\,\},
		\quad w_1:=w_+,\ w_{-1}:=w_-,
		\]
		which has dimension $d$, since $w_+$ and $w_-$ are linearly independent. The inclusion $\supseteq$ is immediate from Theorem~\ref{thm:range}.
		
		For the reverse inclusion, let $w\in \clw\cap\ran(I-S_\theta^{2})$. Since $\{w_+,w_-\}$ spans $\clw$, we may write $w=a\,w_++b\,w_-$ with $a,b\in\C$.  So it suffices to show that the coefficient of $w_\zeta$ vanishes whenever $\theta$ has no ADC at $\zeta$; specifically, $a=0$ unless $\theta$ has an ADC at $1$, and $b=0$ unless $\theta$ has an ADC at $-1$. For the coefficient $a$, we factor $I-S_\theta^{2}=(I-S_\theta)(I+S_\theta)$ and choose $h$ with $(I-S_\theta^{2})h=w=a\,w_++b\,w_-$. Using $w_-=(I-S_\theta)k_0^\theta$,
		\[
		(I-S_\theta)\bigl[(I+S_\theta)h-b\,k_0^\theta\bigr]
		=(I-S_\theta^{2})h-b\,(I-S_\theta)k_0^\theta
		=a\,w_+,
		\]
		so $a\,w_+\in\ran(I-S_\theta)$. If $a\ne0$, then $w_+=(I+S_\theta)k_0^\theta\in\ran(I-S_\theta)$. Thus 
		\[
		2k_0^\theta=w_++(I-S_\theta)k_0^\theta\in\ran(I-S_\theta),
		\]
		that is, $k_0^\theta\in\ran(I-S_\theta)$, which by Lemma~\ref{lem:solve} holds if and only if $\theta$ has an ADC at $1$. Thus, $a\ne0$ implies that $\theta$ has an ADC at $1$.
		
		The statement for $b$ follows from the symmetric computation, with $I-S_\theta^{2}=(I+S_\theta)(I-S_\theta)$ and $w_+=(I+S_\theta)k_0^\theta$. Consequently the coefficient of $w_\zeta$ vanishes whenever $\theta$ has no ADC at $\zeta$, which proves $\subseteq$.
		
		As $I-S_\theta^{2}$ is injective, it maps $\clv_2$
		bijectively onto $\clw\cap\ran(I-S_\theta^{2})$, so $\dim \clv_1=\dim \clv_2=d$; and, since
		$(I-S_\theta^{2})k_\zeta^\theta=w_\zeta$ for $\zeta\in\{1,-1\}$, provided $\theta$ has ADC at $\zeta$, we obtain
		\[
		\clv_1=\spn\{\,k_\zeta^\theta:\zeta\in\{1,-1\},\ \theta\ \text{has an ADC at}\ \zeta\,\}.
		\]
		
		By Proposition~\ref{prop:gens}, $P_\zeta\in\cli_\theta$ whenever $\zeta\in\{1,-1\}$ and $\theta$ has
		an ADC at $\zeta$, so
		\[
		\spn\{\,P_\zeta:\zeta\in\{1,-1\},\ \theta\ \text{has an ADC at}\ \zeta\,\}\subseteq\cli_\theta.
		\]
		We prove the reverse inclusion by considering the three values of $d$.
		
		\emph{Case I.} If $d=2$, then $\theta$ has an ADC at both $1$ and $-1$, and Theorem~\ref{thm:main} gives
		$\cli_\theta=\spn\{P_1,P_{-1}\}$.
		
		\emph{Case II.} If $d\le1$. For $C\in\cli_\theta$, $\ran C\subseteq \clv_1$  and
		$\ran C^{*}\subseteq \clv_1$.
		
		\emph{Subcase I.} If $d=0$ then $\clv_1=\{0\}$, so $C=0$ and
		$\cli_\theta=\{0\}$.
		
		\emph{Subcase II.} If $d=1$, with ADC point $\zeta$, then $\clv_1=\C k_{\zeta}^\theta$. A routine computation shows that $C\in\C P_{\zeta}$.\\
		In either case,
		\[
		\cli_\theta=\spn\{\,P_\zeta:\zeta\in\{1,-1\},\ \theta\ \text{has an ADC at}\ \zeta\,\}.\]
	\end{proof}
	
	\begin{corollary}\label{cor:cases}
		Let $\dim\clk_\theta\ge3$.
		\begin{enumerate}
			\item If $\theta$ has an ADC at both $1$ and $-1$ (in particular for every finite Blaschke product,
			or whenever $1,-1\notin\sigma(S_\theta)$), then $\dim\cli_\theta=2$ and
			$\cli_\theta=\spn\{P_1,P_{-1}\}$.
			\item If $\theta$ has an ADC at exactly one of $\pm1$, then $\dim\cli_\theta=1$; the
			decomposition of any $T\in\cls_\theta$ is unique up to one complex parameter. For  $\theta(z)=\exp\!\bigl(-\frac{1+z}{1-z}\bigr)$,
			for which $\sigma(S_\theta)=\{1\}$, $\theta$ has no ADC at $1$ and continues analytically across
			$-1$, giving $\cli_\theta=\C\,P_{-1}$.
			\item     If $\theta$ has an ADC at neither $1$ nor $-1$, then $\cli_\theta=\{0\}$. In that case, the
			Toeplitz-plus-Hankel decomposition is \emph{unique}. For example $\theta(z)=\exp\!\bigl(-\frac{1+z}{1-z}-\frac{1-z}{1+z}\bigr)$, with no ADC at $1$ and $-1$.
		\end{enumerate}
	\end{corollary}

	\newsection{Algebraic consequences}\label{sec:alg}
	
In view of the complete description of $\cli_\theta$, we proceed to investigate certain algebraic properties of $\cls_\theta$. 
	
	\subsection{Zero decomposition}
	Since $A+B=0$ if and only if $A=-B \in \cli_\theta$, the following is immediate from Theorem \ref{thm:main} and Theorem  \ref{thm:dim}. The case for  $\dim \clk_\theta=1$ is trivial.

\begin{theorem}
Let $T = A +B \in \cls_\theta$. Then $T=0$ if and only if exactly one of the following conditions is satisfied:
		\begin{enumerate}
			\item $\theta$ has an ADC at both $1$ and $-1$ and $A=-B=aP_1 +bP_{-1}$ for $a,b \in \C$.
			\item $\theta$ has an ADC at exactly one of $1$ and $-1$, say $\zeta$, and $A=-B=aP_{\zeta}$ for $a \in \C$.
			\item $\theta$ has an ADC at neither $1$ nor $-1$, and $A=B=0.$
		\end{enumerate}
		
	\end{theorem}

	For the rest of the section, $\theta$ has an ADC at both $\pm1$ and $\dim\clk_\theta\ge2$, so that
	$\cli_\theta=\spn\{P_1,P_{-1}\}$ by Theorem \ref{thm:main}. Note that $P_1$ and $P_{-1}$ are self-adjoint, being of the form
	$p\otimes p$. 
	
	\subsection{Uniqueness of the decomposition}
	
	\begin{theorem}\label{thm:unique}
		Let $T\in\cls_\theta$ and fix one decomposition $T=A+B$ with $A\in\clt_\theta$ and
		$B\in\clh_\theta$. Then the  decompositions of $T$ of this form are precisely
		\[
		T=\bigl(A+c_1P_1+c_2P_{-1}\bigr)+\bigl(B-c_1P_1-c_2P_{-1}\bigr),
		\qquad c_1,c_2\in\C,
		\]
		and distinct pairs $(c_1,c_2)$ give distinct decompositions.
	\end{theorem}
	
	\begin{proof}
		If $T=A'+B'$ is another such decomposition, then 
		\[
		A'-A=B-B'\in \clt_\theta \cap \clh_\theta= \cli_\theta.
		\]
		By Theorem~\ref{thm:main}, it equals $c_1P_1+c_2P_{-1}$ for
		some $c_1,c_2\in\C$, and then $A'=A+c_1 P_1 +c_2 P_{-1}$ and  $B'=B-c_1P_1-c_2P_{-1}$.
		
		Conversely, for any $c_1,c_2\in\C$, the operator $A+c_1P_1+c_2P_{-1}\in\clt_\theta$ and
		$B-c_1P_1-c_2P_{-1}\in \clh_\theta$, by Propositions \ref{prop:TTO} and
		\ref{prop:Hdich}(i), and their sum is $T$. Finally, if two pairs gave the same decomposition
		then $(c_1-c_1')P_1+(c_2-c_2')P_{-1}=0$, so $c_1=c_1'$ and $c_2=c_2'$ by the independence of
		$P_1,P_{-1}$ (Proposition~\ref{prop:gens}).
	\end{proof}
	
	\subsection{Self-adjoint decompositions}
	
	\begin{theorem}\label{thm:selfadj}
		Let $T\in\cls_\theta$ be self-adjoint.
		\begin{enumerate}
			\item There exist self-adjoint $A\in\clt_\theta$ and $B\in\clh_\theta$ with $T=A+B$.
			\item If $T=A'+B'$ is \emph{any} decomposition with $A'\in\clt_\theta$ and $B'\in\clh_\theta$, then
			there are unique $s,t\in\R$ with
			\[
			A'-(A')^{*}=i\bigl(sP_1+tP_{-1}\bigr)=-\bigl(B'-(B')^{*}\bigr),
			\]
			and $A',B'$ are both self-adjoint if and only if $s=t=0$.
			\item With $A,B$ as in \textnormal{(i)}, the  self-adjoint decompositions of $T$ are precisely
			\[
			T=\bigl(A+sP_1+tP_{-1}\bigr)+\bigl(B-sP_1-tP_{-1}\bigr)
			\]
			with $s,t\in\R$.
		\end{enumerate}
	\end{theorem}
	
	\begin{proof}
		(i) Write $T=A_0+B_0$ with $A_0\in\clt_\theta$, $B_0\in\clh_\theta$, and put
		$A=\tfrac12(A_0+A_0^{*})$, $B=\tfrac12(B_0+B_0^{*})$. Then $A\in\clt_\theta$ and
		$B\in\clh_\theta$, both are   self-adjoint, and $A+B=\tfrac12(T+T^{*})=T$.
		
		(ii) Put $D=A'-(A')^{*} \in \clt_\theta$. Then $T=A'+B'$, being self-adjoint, gives
		$D=(B')^{*}-B' \in \clh_\theta$. So,
		$D\in\cli_\theta$ and
		\[
		D=aP_1+bP_{-1},
		\]
		for unique $a,b\in\C$ by Theorem~\ref{thm:main} and
		Proposition~\ref{prop:gens}. Taking adjoints, we get
		\[
		-D=D^{*}=\overline{a}P_1+\overline{b}P_{-1}
		\]
		and comparing coefficients gives
		$\overline{a}=-a$ and $\overline{b}=-b$, that is, $a=is$, $b=it$ with $s,t\in\R$ unique. Finally,
		$A'$ is self-adjoint precisely when $D=0$, i.e.,\ $s=t=0$, and then $B'-(B')^{*}=-D=0$ as well.
		
		(iii) By Theorem~\ref{thm:unique}, every decomposition is $A'=A+c_1P_1+c_2P_{-1}$,
		$B'=B-c_1P_1-c_2P_{-1}$ with $c_1,c_2\in\C$. Since $A$ and $P_{\pm1}$ are self-adjoint,
		$A'-(A')^{*}=(c_1-\overline{c_1})P_1+(c_2-\overline{c_2})P_{-1}$, which vanishes if and only if
		$c_1,c_2$ are real; the same condition makes $B'$ self-adjoint.
	\end{proof}
	
	\begin{remark}\label{rem:sa}
		Part (i) uses only that both classes are closed under adjoints, and would hold whatever
		$\cli_\theta$ turned out to be; it is part (ii) that requires Theorem~\ref{thm:main}, the  identification of $\cli_\theta$ being what pins the deviation $A'-(A')^{*}$ down to the  two-real-parameter family $i\spn_{\R}\{P_1,P_{-1}\}$. That it can indeed fail is seen at once:   for $s\in\R\setminus\{0\}$ the operators $A'=isP_1\in \clt_\theta$ and $B'=-isP_1\in\clh_\theta$
		are non-zero and skew-Hermitian, yet $A'+B'=0$ is self-adjoint. In the setting of
		Corollary~\ref{cor:cases}(ii), the same statements hold with the single parameter $sP_\zeta$
		($s\in\R$); in the setting of (iii), every decomposition of a self-adjoint operator has
		self-adjoint parts.
	\end{remark}
	
	\subsection{An intrinsic inclusion criterion}
	
	The explicit description of $\cli_\theta$ yields a symbol-free test for inclusion in $\cls_\theta$, which, in turn, becomes a novel characterization of $\cls_\theta$,
	refining the displacement calculus of \cite{BM}. Recall the space
	\[
	\clm=\{k_0^\theta\otimes f+\wt{k}_0^\theta\otimes g+u\otimes k_0^\theta+v\otimes S_\theta k_0^\theta:f,g,u,v\in \clk_\theta\}
	\]
	from the introduction, and write
	\[
	\clz_\theta:=\{k_0^\theta\otimes f-\wt{k}_0^\theta\otimes g:f,g\in \clk_\theta\}.
	\]
	By \eqref{eq:comm}, $\clz_\theta$ is the subspace in which the commutator $[D_H,D_T]$ takes its values.
	
	\begin{lemma}\label{lem:env}
		For every inner function $\theta$, we have $\ D_HD_T(\cls_\theta)\subseteq\clm$. If $\theta$ is a finite
		Blaschke product of degree $n\ge2$, then $\dim\clm=4(n-1)$.
	\end{lemma}
	
	\begin{proof}
		Let $T=A+B$ with $A\in\clt_\theta$ and $B\in\clh_\theta$. Then
		\[
		D_T(A)\in\clr_T \implies D_HD_T(A)\in D_H(\clr_T),
		\]
		while $D_HD_T=D_TD_H+[D_H,D_T]$, together with $D_H(B)\in\clr_H$ and
		\eqref{eq:comm}, gives
		\[
		D_HD_T(B)\in D_T(\clr_H)+\clz_\theta.
		\]
		Each of $D_H(\clr_T)$, $D_T(\clr_H)$ and $\clz_\theta$ is contained in $\clm$. Indeed, by \eqref{eq:s3},
		\[
		D_H(u\otimes k_0^\theta)=u\otimes k_0^\theta-(S_\theta^{*}u)\otimes(S_\theta k_0^\theta)
		\]
		and
		\[
		D_H(k_0^\theta\otimes v)=k_0^\theta\otimes v+\overline{\theta(0)}\,(\wt{k}_0^\theta\otimes S_\theta v),
		\]
		and similarly for $D_T(\clr_H)$ using \eqref{eq:s4}, while $\clz_\theta\subseteq\clm$ is clear. Thus 
		\[
		D_HD_T(T)=D_HD_T(A)+D_HD_T(B)\in\clm.
		\]
		
		For the bound, write $\clm=\clm_1+\clm_2$ with
		\[
		\clm_1=\{X \in \clb(\clk_\theta):\ran X\subseteq\spn\{k_0^\theta,\wt{k}_0^\theta\}\}
		\]
		and
		\[
		\clm_2=\{X \in \clb(\clk_\theta):\ran X^{*}\subseteq\spn\{k_0^\theta,S_\theta k_0^\theta\}\}.
		\]
		For $n\ge2$, $\clm_1$ and $\clm_2$ have dimension $2n$ each (Lemma~\ref{lem:cyclic}), while $\clm_1\cap \clm_2$ is spanned by the four independent operators $p\otimes q$ with $p\in\{k_0^\theta,\wt{k}_0^\theta\}$ and $q\in\{k_0^\theta,S_\theta k_0^\theta\}$, so that $\dim(\clm_1\cap \clm_2)=4$. Hence
		\[
		\dim\clm=\dim \clm_1+\dim \clm_2-\dim(\clm_1\cap \clm_2)=2n+2n-4=4(n-1). \qedhere
		\]
	\end{proof}
	
	\begin{theorem}\label{thm:criterion}
		Let $\theta$ be a finite Blaschke product with $n=\dim\clk_\theta\ge2$. Then $\dim\cls_\theta=4(n-1)$
		and, for every $T\in\clb(\clk_\theta)$,
		\[
		T\in\cls_\theta\iff D_HD_T(T)\in\clm .
		\]
	\end{theorem}
	
	\begin{proof}
		By Theorem~\ref{thm:main}, $\dim\cli_\theta=2$; with $\dim\clt_\theta=\dim\clh_\theta=2n-1$
		(Section~\ref{sec:prelim}), we have $\dim\cls_\theta=(2n-1)+(2n-1)-2=4(n-1)$. The
		forward implication is established in Lemma~\ref{lem:env}. For the converse, $\Phi:=D_HD_T$ is injective, hence a
		bijection on $\clb(\clk_\theta)$, so $\dim\Phi^{-1}(\clm)=\dim\clm=4(n-1)$ by
		Lemma~\ref{lem:env}. Since $\cls_\theta\subseteq\Phi^{-1}(\clm)$ and both have dimension $4(n-1)$,
		they coincide. This finishes the proof.
	\end{proof}
	
	\begin{remark}\label{rem:solvability}
		Theorem~\ref{thm:criterion} makes the criterion of \cite[Theorem 5.6]{BM} intrinsic. Whereas the   latter characterizes membership in $\cls_\theta$ by the solvability of an auxiliary operator  equation, Theorem~\ref{thm:criterion} imposes the single condition $D_HD_T(T)\in\clm$ on  $T$ directly. Moreover, since $[D_H,D_T]$ takes values in $\clz_\theta \subseteq \clm$ by \eqref{eq:comm}, one has $D_T D_H(T)\in\clm$ if and only if $D_HD_T(T)\in\clm$, so $D_T D_H$ yields the same characterization.
		For $\theta=z^{n}$, dimension of $\cls_\theta$ is $4(n-1)$, the dimension of the space of  $n\times n$ Toeplitz-plus-Hankel matrices. 
	\end{remark}

	\begin{example}
		For $\theta=z^3$, $\dim\clb(\clk_\theta)=9$ and $\dim \cls_\theta=8$. Then the dimension of the quotient space $\clb(\clk_\theta)/\cls_\theta$ is 1. By inspection, we  get
		\[
		\clb(\clk_\theta)/\cls_\theta=\spn \left\{
		\begin{bmatrix}
			0 & 1 & 0\\
			0 & 0 & 0\\
			0 & 1 & 0			
		\end{bmatrix} + \cls_\theta \right\}.
		\]
		To see this, we just have to prove that $E=\begin{bmatrix}
			0 & 1 & 0\\
			0 & 0 & 0\\
			0 & 1 & 0			
		\end{bmatrix} \notin \cls_\theta$. Let, if possible, there exist $A=(a_{ij}) \in \clt_\theta$ and $B=(b_{ij}) \in \clh_\theta$ such that $E=A+B$. (Clearly, $a_{ij}=a_{i+l,j+l}$ for $1\leq i+l,j+l \leq 3$ and $b_{ij}=b_{i+l,j-l}$ for $1\leq i+l,j-l \leq 3$.) Then, comparing entries yields $a_{12}+b_{12}=1$ which implies $b_{12}=1-a_{12}$ and hence $b_{21}=1-a_{12}$. Similarly, $b_{23}=b_{32}=1-a_{32}$. But, we also have $b_{21}=-a_{21}=-a_{32}$ and $b_{23}=-a_{23}=-a_{12}$. Combining them together, we get $1-a_{12}+a_{32}=0$ and $1-a_{32}+a_{12}=0$ which gives a contradiction.
	\end{example}

	\newsection{A commutator characterization and products}\label{sec:comm}
	
	The criterion of Theorem~\ref{thm:criterion} is of second order because it tests the operator $D_HD_T(X)$. In this section we show that $\cls_\theta$ admits a first-order description, through a single commutator with the help of a  self-adjoint operator
	\[
	\Sig:=S_\theta+S_\theta^{*}.
	\]

	We use the following commutator characterizations of the two classes, which follow from
	\eqref{eq:TTOchar}--\eqref{eq:THOchar} on multiplying by $S_\theta$ and $S_\theta^{*}$ and using
	the identities \eqref{eq:s1}--\eqref{eq:s4} (see \cite[Lemmas 5.1--5.2]{BM}).
	
	\begin{lemma}[\cite{Sarason,BM}]\label{lem:Tcomm}
		For $A\in\clb(\clk_\theta),$ the following are equivalent:
		\begin{enumerate}
			\item $A\in\clt_\theta$
			\item $S_\theta A-A S_\theta=x\otimes\wt{k}_0^\theta+k_0^\theta\otimes y$ for some $x,y\in\clk_\theta$.
			\item $S_\theta^{*}A-A S_\theta^{*}=u\otimes k_0^\theta+\wt{k}_0^\theta\otimes v$ for some $u,v\in\clk_\theta$.
		\end{enumerate}
	\end{lemma}
	
	\begin{lemma}[\cite{GuMa,BM}]\label{lem:Hcomm}
		For $B\in\clb(\clk_\theta),$ the following are equivalent:
		\begin{enumerate}
			\item $B\in\clh_\theta$
			\item $S_\theta B-B S_\theta^{*}=x\otimes k_0^\theta+k_0^\theta\otimes y$ for some $x,y\in\clk_\theta$.
			\item $S_\theta^{*}B-B S_\theta=u\otimes\wt{k}_0^\theta+\wt{k}_0^\theta\otimes v$ for some $u,v\in\clk_\theta$.
		\end{enumerate}
	\end{lemma}
	
	The four right-hand sides occurring in Lemmas~\ref{lem:Tcomm} and \ref{lem:Hcomm} all lie in the
	single space
	\begin{equation}\label{eq:Rsigma}
		\clr_\Sigma:=\bigl\{f\otimes k_0^\theta+k_0^\theta\otimes g+u\otimes\wt{k}_0^\theta
		+\wt{k}_0^\theta\otimes v:f,g,u,v\in\clk_\theta\bigr\},
	\end{equation}
	the symmetrisation of $\clr_T$ and $\clr_H$. It admits a corner description analogous to	Lemma~\ref{lem:corner}. Write  $\cle:=\spn\{k_0^\theta,\wt{k}_0^\theta\}$,
	if $\dim\clk_\theta\ge2$ then, by Lemma~\ref{lem:cyclic}, it is two-dimensional.
	
	\begin{lemma}\label{lem:Rsig}
		Let $Q$ be the orthogonal projection of $\clk_\theta$ onto $\cle^{\perp}$. Then, for
		$X\in\clb(\clk_\theta)$, 
		\[
		\ X\in\clr_\Sigma\iff QXQ=0.
		\]
		If $\dim\clk_\theta=n<\infty$, then
		$\dim\clr_\Sigma=4(n-1)$.
	\end{lemma}
	
	\begin{proof}
		If 
		\[
		X=f\otimes k_0^\theta+k_0^\theta\otimes g+u\otimes\wt{k}_0^\theta+\wt{k}_0^\theta\otimes v,
		\]
		then, each of the four terms of $QXQ$ contains a factor $Qk_0^\theta$ or 	$Q\wt{k}_0^\theta$, so $QXQ=0$.

		Conversely, suppose $QXQ=0$, and let $\{e_1,e_2\}$ be an  orthonormal basis of $\cle$, so that  $I-Q=e_1\otimes e_1+e_2\otimes e_2$. Since $\ran XQ\subseteq\ker Q=\cle$, we have $XQ=e_1\otimes y_1+e_2\otimes y_2$ for some $y_1,y_2 \in \clk_\theta$, while  $X(I-Q)=(Xe_1)\otimes e_1+(Xe_2)\otimes e_2$. Adding and re-expressing $e_1,e_2$ in terms of  $k_0^\theta,\wt{k}_0^\theta$ exhibits $X\in\clr_\Sigma$. For the dimension, note that the map  $X\mapsto QXQ$ from $\clb(\clk_\theta)$ onto $\clb(\cle^{\perp})$ is surjective with kernel
		$\clr_\Sigma$, so $\dim\clr_\Sigma=n^{2}-(n-2)^{2}=4(n-1)$.
	\end{proof}
	
	Note that $\dim\clr_\Sigma=4(n-1)=\dim\cls_\theta$ by Theorem~\ref{thm:criterion}, a first indication
	that $\clr_\Sigma$ is the right target space.
	
	\begin{proposition}\label{prop:necessity}
		Let $\theta$ be any inner function and $X\in\cls_\theta$. Then $\Sig X-X\Sig\in\clr_\Sigma$.
	\end{proposition}
	
	\begin{proof}
		By linearity, it suffices to treat $X=A\in\clt_\theta$ and $X=B\in\clh_\theta$ separately. For
		$A\in\clt_\theta$,
		\[
		\Sig A-A\Sig=\bigl(S_\theta A-A S_\theta\bigr)+\bigl(S_\theta^{*}A-A S_\theta^{*}\bigr),
		\]
		and the two brackets are in $\clr_\Sigma$ by Lemma~\ref{lem:Tcomm}(ii),(iii). For
		$B\in\clh_\theta$ the same expansion can be written as 
		\[
		\Sig B-B\Sig=\bigl(S_\theta B-B S_\theta^{*}\bigr)+\bigl(S_\theta^{*}B-B S_\theta\bigr),
		\]
		and the two bracket terms  are in $\clr_\Sigma$ by Lemma~\ref{lem:Hcomm}(ii),(iii). Since
		$\clr_\Sigma$ is a linear subspace, the proposition follows.
	\end{proof}

	The natural question is to check whether the converse of the above proposition is true or not. We give an affirmative answer to it in the case of a finite Blaschke product $\theta$. In that case, the following theorem also provides yet another characterization of elements of $\cls_\theta$:

	\begin{theorem}\label{thm:comm}
		Let $\theta$ be a finite Blaschke product of degree $n\ge2$ and $X\in\clb(\clk_\theta)$. Then
		\[
		X\in\cls_\theta\iff\Sig X-X\Sig\in\clr_\Sigma\iff Q\bigl(\Sig X-X\Sig\bigr)Q=0,
		\]
		where $Q$ is the orthogonal projection onto $\cle^{\perp}$.
	\end{theorem}
	
	\begin{proof}
		The second equivalence is Lemma~\ref{lem:Rsig} and the necessity of the first equivalence is
		Proposition~\ref{prop:necessity}. It remains to prove the sufficiency. Suppose
		\[
		\Sig X-X\Sig=f\otimes k_0^\theta+k_0^\theta\otimes g+u\otimes\wt{k}_0^\theta+\wt{k}_0^\theta\otimes v.
		\]
		Multiply on the right by $S_\theta^{*}$  and using \eqref{eq:s4},
		\begin{equation}\label{eq:multS}
			\bigl(\Sig X-X\Sig\bigr)S_\theta^{*}
			=f\otimes\bigl(S_\theta k_0^\theta\bigr)+k_0^\theta\otimes\bigl(S_\theta g\bigr)
			-\overline{\theta(0)}\,\bigl(u\otimes k_0^\theta\bigr)+\wt{k}_0^\theta\otimes\bigl(S_\theta v\bigr).
		\end{equation}
		On the left, expanding $\Sig =S_\theta +S_\theta^*$ and  using \eqref{eq:s1},
		\[
		\bigl(\Sig X-X\Sig\bigr)S_\theta^{*}
		=-\bigl(X-S_\theta X S_\theta^{*}\bigr)-\bigl(X(S_\theta^{*})^{2}-S_\theta^{*}X S_\theta^{*}\bigr)
		+\bigl(Xk_0^\theta\bigr)\otimes k_0^\theta.
		\]
		Comparing with \eqref{eq:multS} and collecting terms, there exist $f_1,g_1,u_1,v_1\in\clk_\theta$
		with
		\begin{equation}\label{eq:mid}
			\bigl(X-S_\theta X S_\theta^{*}\bigr)+\bigl(X(S_\theta^{*})^{2}-S_\theta^{*}X S_\theta^{*}\bigr)
			=f_1\otimes\bigl(S_\theta k_0^\theta\bigr)+k_0^\theta\otimes g_1+u_1\otimes k_0^\theta
			+\wt{k}_0^\theta\otimes v_1.
		\end{equation}
		On the other hand, 	\eqref{eq:s2} yields
		\begin{align*}
			D_HD_T(X) & = (X-S_\theta XS_\theta^*)- S_\theta^* (X-S_\theta XS_\theta^*)S_\theta^*\\
			&=  \bigl(X-S_\theta X S_\theta^{*}\bigr)+\bigl(X(S_\theta^{*})^{2}-S_\theta^{*}X S_\theta^{*}\bigr)
			-\wt{k}_0^\theta\otimes\bigl(S_\theta^{2}X^{*}\wt{k}_0^\theta\bigr).
		\end{align*}
Substituting \eqref{eq:mid}, we get
		\[
		D_HD_T(X)=f_1\otimes\bigl(S_\theta k_0^\theta\bigr)+k_0^\theta\otimes g_1+u_1\otimes k_0^\theta
		+\wt{k}_0^\theta\otimes\bigl(v_1-S_\theta^{2}X^{*}\wt{k}_0^\theta\bigr),
		\]
		which lies in $\clm$. By Theorem~\ref{thm:criterion}, $X\in\cls_\theta$.
	\end{proof}
	
	\begin{remark}\label{rem:trivial}
		For $n=2$ the statement is true but empty, since there $\dim\clr_\Sigma=4(n-1)=4=n^{2}$, so
		$\clr_\Sigma=\clb(\clk_\theta)$ and the condition holds for every $X$; correspondingly
		$\dim\cls_\theta=4=n^{2}$, so $\cls_\theta=\clb(\clk_\theta)$. The content of
		Theorem~\ref{thm:comm} is therefore in the range $n\ge3$.
	\end{remark}

	\begin{remark}
		It is clear from Theorem \ref{thm:comm} that if a bounded operator on a finite-dimensional model space  $\clk_\theta$ commutes with $\Sig$, then it must be in $\cls_\theta$. This leads us to the following complete description of the operators in $\cls_\theta$ commuting with $\Sig$. We will use the following lemma in our next result.
	\end{remark}

		\begin{lemma}[{cf.~\cite[Lemm 4.3]{BM}}]\label{lem:auxproj}
			For every inner function $\theta$ and every $f\in\clk_\theta$,
			\[
		P_{\theta H^2}{(zf)}=\ip{f}{\wt{k}_0^\theta}\,\theta.
			\]
		\end{lemma}

\begin{theorem}\label{thm:commutant}
Let $\theta$ be a finite Blaschke product of degree $n\ge2$. Then every bounded operator on $\clk_\theta$ commuting with $\Sig=S_\theta+S_\theta^{*}$ is a truncated
Toeplitz-plus-Hankel operator. Moreover, $\Sig$ has $n$ simple eigenvalues, all in $(-2,2)$, and
		\[
		\{X\in\cls_\theta:\Sig X=X\Sig\}=\{\Sig\}'=\bigl\{p(\Sig):p\in\C[z]\bigr\}=\spn\{I,\Sig,\dots,\Sig^{\,n-1}\},
		\]
a commutative subalgebra of $\cls_\theta$ of dimension $n$.
\end{theorem}

\begin{proof}
		If $X$ commutes with $\Sig$ then $\Sig X-X\Sig=0\in\clr_\Sigma$, so $X\in\cls_\theta$ by
		Theorem~\ref{thm:comm}, $\{\Sig\}'\subseteq\cls_\theta$.
		
		For the spectrum, let $\Sig f=\lambda f$ with $f\in\clk_\theta\setminus\{0\}$. If $\lambda=\pm2$ then
		$\operatorname{Re}\ip{S_\theta f}{f}=\pm\norm{f}^{2}$ which, by Cauchy--Schwarz, forces $S_\theta f=\pm f$.
		This is impossible since $S_\theta$ is of class $C_{00}$, so $\lambda\in(-2,2)$. For $f\in\clk_\theta$,   $S_\theta^{*}f=z^{-1}\bigl(f-f(0)\bigr)$ and by Lemma~\ref{lem:auxproj},
		\[
		S_\theta f=(P_+-P_{\theta H^2})zf=zf-\ip{f}{\wt{k}_0^\theta}\theta.
		\]
		Adding these and multiplying by $z$,
		\[
		z^2f-\ip{f}{\wt{k_0^\theta}}z\theta+f-f(0)=\lambda z f,		
		\qquad\text{that is,}\qquad
		f(z)=\frac{c\,z\theta(z)+d}{z^{2}-\lambda z+1},
		\]
		with $c=\ip{f}{\wt{k}_0^\theta}$ and $d=f(0)$. Write $\lambda=2 \cos \alpha=\zeta+\overline{\zeta}$ with
		$\zeta=e^{i\alpha}\in\T$. Since $\lambda\ne\pm2$, the zeros $\zeta,\overline{\zeta}$ of
		$z^{2}-\lambda z+1$ are distinct. As $\theta$ is analytic across $\T$ while
		$(1-\overline{\zeta}z)^{-1}\notin H^{2}$ for $|\zeta|=1$, an $H^{2}$ function of the above form must have
		vanishing numerator at both zeros --- otherwise, by partial fractions, $f$ would inherit a non-$H^{2}$
		Cauchy kernel at $\zeta$ or $\overline{\zeta}$. Thus
		\[
		c\,\zeta\theta(\zeta)+d=0,\qquad c\,\overline{\zeta}\theta(\overline{\zeta})+d=0,
		\]
		which  gives $c\bigl(\zeta\theta(\zeta)-\overline{\zeta}\theta(\overline{\zeta})\bigr)=0$. If
		$\zeta\theta(\zeta)\ne\overline{\zeta}\theta(\overline{\zeta})$ then $c=d=0$ and $f=0$, a contradiction.
		Hence $\zeta\theta(\zeta)=\overline{\zeta}\theta(\overline{\zeta})$ and $d=-c\,\zeta\theta(\zeta)$ is
		determined by $c$, so the eigenspace is one-dimensional. Therefore every eigenvalue of the self-adjoint operator
		$\Sig$ is simple. A self-adjoint operator on an $n$-dimensional space with simple spectrum is
		non-derogatory, so its commutant equals $\{p(\Sig):p\in\C[z]\}$, of dimension $n$ (\cite[Chapter 3]{Horn}). Combined with
		$\{\Sig\}'\subseteq\cls_\theta$, this yields the stated  assertion.
	\end{proof}
	
	\subsection{Application to products}\label{subsec:products}
	
	Our next goal is to address when a product of two elements of $\cls_\theta$ again lies in $\cls_\theta$. Note that 
	\begin{equation}\label{eq:leibniz}
		\Sig(XY)-(XY)\Sig =\bigl(\Sig X-X \Sig\bigr)Y +X \bigl(\Sig Y-Y\Sig\bigr). 
	\end{equation}
	For the rest of the section, we will denote $R_X=\Sigma _\theta X-X\Sigma_\theta$ for $X \in \clb(\clk_\theta)$. In the following theorem, $\{e_1,e_2\}$ is the Gram-Schmidt orthogonalization of the linearly independent set $\{k_0^\theta, \wt{k}_0^\theta\}$.
	
	\begin{theorem}\label{thm:product}
		Let $\theta$ be a finite Blaschke product of degree $n\ge2$. Let $A,B\in\cls_\theta$ and
		$\cle=\spn\{k_0^\theta,\wt{k}_0^\theta\}$ and $Q$ the orthogonal projection onto $\cle^{\perp}$. Then
		\begin{equation}\label{eq:prodformula}
			Q\bigl(\Sig(AB)-(AB)\Sig\bigr)Q
			=\sum_{i=1}^{2}\Bigl[\bigl(QR_Ae_i\bigr)\otimes\bigl(QB^{*}e_i\bigr)
			+\bigl(QAe_i\bigr)\otimes\bigl(QR_B^{*}e_i\bigr)\Bigr],
		\end{equation}
		and consequently
		\[
		AB\in\cls_\theta\iff\sum_{i=1}^{2}\Bigl[\bigl(QR_Ae_i\bigr)\otimes\bigl(QB^{*}e_i\bigr)
		+\bigl(QAe_i\bigr)\otimes\bigl(QR_B^{*}e_i\bigr)\Bigr]=0 .
		\]
		
	\end{theorem}
	
	\begin{proof}
		By \eqref{eq:leibniz}, we must compute $Q\bigl(R_AB+A R_B\bigr)Q$. For the first term, write:
		\[
		QR_ABQ=QR_A(Q+I-Q)BQ.
		\]
		Since $A\in\cls_\theta$,
		Lemma~\ref{lem:Rsig} and Proposition~\ref{prop:necessity} give $QR_AQ=0$, so the first summand
		vanishes. Writing $I-Q=e_1\otimes e_1+e_2\otimes e_2$ gives
		\[
		QR_A(I-Q)BQ=\sum_{i=1}^2 QR_A\bigl(e_i\otimes e_i\bigr)BQ=\sum_{i=1}^2\bigl(QR_Ae_i\bigr)\otimes\bigl(QB^{*}e_i\bigr).
		\]
		For the second term, $QR_BQ=0$ gives $\ran(R_BQ)\subseteq \cle$, so $R_BQ=(I-Q)R_BQ$ and therefore
		\[
		QAR_BQ=QA(I-Q)R_BQ=\sum_{i=1}^2\bigl(QAe_i\bigr)\otimes\bigl(QR_B^{*}e_i\bigr).
		\]
		Adding both equations gives \eqref{eq:prodformula}. The stated equivalence is Theorem~\ref{thm:comm} applied to
		$X=AB$.
	\end{proof}
	
	In particular, Theorem \ref{thm:product} establishes that the inclusion of $AB$ in $\cls_\theta$ is decided by a sum of at most four rank-one operators, determined by the action of $A,B,R_A,R_B$ on the two-dimensional space $\cle$. Reading off the conditions under which the right-hand side of \eqref{eq:prodformula} vanishes for
	trivial reasons gives a class of universal multipliers.
	
	\begin{corollary}\label{cor:multiplier}
		Let $\theta$ be a finite Blaschke product of degree $n\ge2$ and let $A\in\cls_\theta$ satisfy
		\[
		A\cle\subseteq \cle\qquad\text{and}\qquad R_A \cle\subseteq \cle .
		\]
		Then $AB\in\cls_\theta$ for every $B\in\cls_\theta$. Symmetrically, if $B\in\cls_\theta$ satisfies
		$B^{*}\cle\subseteq \cle$ and $R_B^*\cle\subseteq \cle$, then $AB\in\cls_\theta$ for
		every $A\in\cls_\theta$.
	\end{corollary}
	
	\begin{proof}
		The two hypotheses say precisely that $QAe_i=0$ and $QR_Ae_i=0$ for $i=1,2$, so every term on the right of \eqref{eq:prodformula} vanishes; now apply Theorem~\ref{thm:product}. For the second statement, the hypotheses on $B$ say that $QB^{*}e_i=0$ and $QR_B^{*}e_i=0$ for $i=1,2$, so again every term vanishes.
	\end{proof}
	
	\begin{remark}\label{rem:multiplier}
		The class described in Corollary~\ref{cor:multiplier} always contains the identity. Its precise
		description, and whether it exhausts the operators $A$ with $AB\in\cls_\theta$ for every $B\in\cls_\theta$, remain open.
	\end{remark}

	\newsection{Concluding remarks}\label{sec:Conclusion}
	
	\subsection{The infinite-dimensional case}\label{subsec:converse}
	
	For a finite Blaschke product, Theorem~\ref{thm:comm} characterizes $\cls_\theta$ among all bounded
	operators. We now investigate whether this persists when $\dim\clk_\theta=\infty$, and the remainder
	of this section records what we can establish, closing with an open question. Because the computation
	in the proof of Theorem~\ref{thm:comm} is purely algebraic, one inclusion survives for every inner
	function. The reverse inclusion is the crucial step and needs particular attention for arbitrary inner function. We prove that this inclusion may fail for inner functions without an	ADC at $1$ which points to a negative answer.
	
	\begin{proposition}\label{prop:reduction}
		For every inner function $\theta$ and every $X\in\clb(\clk_\theta)$,
		\[
		\Sig X-X\Sig\in\clr_\Sigma\ \Longrightarrow\ D_HD_T(X)\in\clm .
		\]
		Consequently, writing $\clc_\theta=\{X:\Sig X-X\Sig\in\clr_\Sigma\}$ and
		$\cld_\theta=\{X:D_HD_T(X)\in\clm\}$, one has, for every inner function,
		\[
		\cls_\theta\subseteq\clc_\theta\subseteq\cld_\theta,
		\]
		with equality throughout when $\dim\clk_\theta<\infty$.
	\end{proposition}
	
	\begin{proof}
		The inclusions are  already established in the proof of Proposition~\ref{prop:necessity} and Theorem~\ref{thm:comm}.  When $\dim\clk_\theta<\infty$, $\cld_\theta=\cls_\theta$ by Theorem~\ref{thm:criterion} forcing equality.
	\end{proof}
	
	\begin{remark}\label{rem:infinite}
		Whether the reverse inclusion $\cld_\theta\subseteq\cls_\theta$ 
		holds for $\dim\clk_\theta=\infty$ is open; it would imply  $\clc_\theta=\cls_\theta$. We have the following observations:
		\begin{enumerate}
			\item $\clc_\theta$ is norm-closed, so if $\clc_\theta=\cls_\theta$, then $\cls_\theta=\clt_\theta+\clh_\theta$ is	norm-closed --- which is not automatic for a sum of two closed subspaces.
			\item  For the finite-dimensional case, $\cld_\theta=\cls_\theta$ rests on the count of Theorem~\ref{thm:criterion}, which uses that $D_H$ and $D_T$, being injective, are automatically surjective, so $\ran D_H=\ran D_T=\clb(\clk_\theta)$.	In particular $\clr_H\subseteq\ran D_H$, so that, for  every $\rho\in\clr_H$,  the	\emph{Stein equation} $B-S_\theta^{*}B S_\theta^{*}=\rho$ (a discrete Lyapunov-type equation,	here with $S_\theta^{*}$ on both sides in place of an adjoint pair; see \cite{HJ} for such equations) has a solution, and symmetrically $\clr_T\subseteq\ran D_T$ recovers the Toeplitz part. Bijectivity of $D_T$ and $D_H$  is precisely what proves the sufficient part of Theorem~\ref{thm:criterion}. In infinite dimensional case, $D_H$ and $D_T$ remain injective \cite[Lemma 5.3]{BM}, but injectivity no longer forces surjectivity, so the inclusions	$\clr_H\subseteq\ran D_H$ and $\clr_T\subseteq\ran D_T$ are not automatic, and thus  the proof of Theorem \ref{thm:criterion} does not carry over; deciding for which inner functions these inclusions hold would settle whether the criterion extends.
		\end{enumerate}
	\end{remark}
	
	The following lemma is well-known and its stronger version can be found in \cite{GuMa}. But we are giving a direct proof which is useful for our next result.
	
	\begin{lemma}\label{lem:RHclosed}
		$\clr_H$ is a norm-closed subspace of $\clb(\clk_\theta)$. Consequently 	$\clh_\theta=\{B:D_H(B)\in\clr_H\}$ is norm-closed as well.
	\end{lemma}
	\begin{proof}
		Let $A \in \clb(\clk_\theta)$ such that $A_n:=x_n\otimes k_0^\theta+\wt{k}_0^\theta\otimes y_n\to A$ in operator norm. Set $\Pi$
		to 	be the orthogonal projection of $\clk_\theta$ onto $\C k_0^\theta$. For $h\perp k_0^\theta$, one has
		$A_nh =\ip{h}{y_n}\wt{k}_0^\theta\in\C\wt{k}_0^\theta$, so
		$A h=\lim\limits_nA_n h\in\C\wt{k}_0^\theta$. Hence $A(I-\Pi)=\wt{k}_0^\theta\otimes y$ for some
		$y\in\clk_\theta$. Since $A\Pi=(A k_0^\theta)\otimes k_0^\theta/\norm{k_0^\theta}^{2}$ is of the
		form $x\otimes k_0^\theta$, we get 
		\[
		A=x\otimes k_0^\theta+\wt{k}_0^\theta\otimes y\in\clr_H.
		\]
		The
		final assertion follows because $\clh_\theta=D_H^{-1}(\clr_H)$ and $D_H$ is continuous.
	\end{proof}
	
	\begin{theorem}\label{thm:steinfail}
		Let $\theta$ be an inner function having no ADC at the point~$1$. Then
		\[
		\clr_H\not\subseteq\ran D_H .
		\]
	\end{theorem}
	\begin{proof}
		Fix $r\in(0,1)$. By \cite[Lemma 2.2]{Sarason} at $\lambda=r\in \R$,
		$S_\theta^{*}k_r^\theta=r\,k_r^\theta-\overline{\theta(r)}\,\wt{k}_0^\theta$ and
		$S_\theta k_r^\theta=r^{-1}(k_r^\theta-k_0^\theta)$. So,
		\[
		S_\theta^{*}\bigl(k_r^\theta\otimes k_r^\theta\bigr)S_\theta^{*}
		=\bigl(S_\theta^{*}k_r^\theta\bigr)\otimes\bigl(S_\theta k_r^\theta\bigr)
		=k_r^\theta\otimes(k_r^\theta-k_0^\theta)-\frac{\overline{\theta(r)}}{r}\,\wt{k}_0^\theta\otimes(k_r^\theta-k_0^\theta).
		\]
		Subtracting from $k_r^\theta\otimes k_r^\theta$  gives
		\begin{equation}\label{eq:steinDH}
			D_H\bigl(k_r^\theta\otimes k_r^\theta\bigr)
			=k_r^\theta\otimes k_0^\theta+\frac{\overline{\theta(r)}}{r}\,\wt{k}_0^\theta\otimes(k_r^\theta-k_0^\theta)
			\ \in\ \clr_H,
		\end{equation}
		hence
		$k_r^\theta\otimes k_r^\theta\in\clh_\theta$ by \eqref{eq:THOchar}. Now $\norm{k_r^\theta\otimes k_r^\theta}=\norm{k_r^\theta}^{2}$. Since $\theta$ has no ADC at $1$,
		\[
		\norm{k_r^\theta}^{2}=k_r^\theta(r)=\frac{1-|\theta(r)|^{2}}{1-r^{2}}\rightarrow \infty \quad \text{as} \quad r \rightarrow 1^-,
		\]
		so that $\norm{k_r^\theta}\ge\norm{k_0^\theta}$ for $r$ near $1$. Hence, using the relations  $\norm{\wt{k}_0^\theta}=\norm{k_0^\theta}$ and $\norm{k_r^\theta-k_0^\theta}\le\norm{k_r^\theta}+\norm{k_0^\theta}\le2\norm{k_r^\theta}$,  equation  \eqref{eq:steinDH} yields 
		\[
		\norm{D_H(k_r^\theta\otimes k_r^\theta)}
		\le\norm{k_0^\theta}\,\norm{k_r^\theta}
		+\frac{|\theta(r)|}{r}\,\norm{k_0^\theta}\bigl(\norm{k_r^\theta}+\norm{k_0^\theta}\bigr)
		\le C\,\norm{k_r^\theta},
		\]
		for $r$ near $1$. Therefore,
		\[
		\frac{\norm{D_H(k_r^\theta\otimes k_r^\theta)}}{\norm{k_r^\theta\otimes k_r^\theta}}\le\frac{C}{\norm{k_r^\theta}}\rightarrow  0 \quad \text{as} \quad r \rightarrow 1^-. 
		\]
		so $D_H$ is not bounded below on $\clh_\theta$.
		
		Suppose, for contradiction, that $\clr_H\subseteq\ran D_H$. Since $B\in\clh_\theta\iff D_H(B)\in\clr_H$
		and $D_H$ is injective \cite[Lemma 5.3]{BM}, $D_H$ maps $\clh_\theta$ bijectively onto $\clr_H$. Both
		are Banach spaces by Lemma~\ref{lem:RHclosed}, so the open mapping theorem makes the inverse of
		$D_H\colon\clh_\theta\to\clr_H$ bounded, that is, $D_H$ is bounded below on $\clh_\theta$,
		contradicting the previous paragraph. Hence $\clr_H\not\subseteq\ran D_H$.
	\end{proof}
	
	Thus the Stein solvability of Remark~\ref{rem:infinite}, on which the finite-dimensional identity
	$\clc_\theta=\cls_\theta$ rests, fails for such $\theta$, and the argument of Theorem~\ref{thm:comm}
	does not apply to them. This leads us to the following question.
	
	\emph{Open Question:}
	For which inner function $\theta$ with $\dim \clk_\theta=\infty$ is $\cls_\theta\subsetneq\cld_\theta$, that is, for which does the characterization of Theorem \ref{thm:comm} fail? In particular, does this happen whenever $\theta$ has no ADC at $1$?

	\vspace{0.4cm}	
	
	\NI\textit{AI Declaration:} 
	The results and proofs presented in this paper are original and were not generated 
	by AI.
	
	\NI\textit{Funding:} 
	No funding is applicable to this project.
	\vspace{0.2cm}
	
	\NI\textit{Data availability:}
	Data sharing is not applicable to this article as no data sets
	were generated or analysed during the current study.
	\vspace{0.2cm}

	\NI\textit{Conflict of interest:}
	The authors have no competing interests to declare.

\end{document}